\documentclass[12pt,a4paper]{amsart}
\pdfoutput=1

\usepackage{a4wide}
\usepackage{amsthm, amsmath, amssymb, graphicx, caption,color,enumerate, tikz, subfig, algorithmic}

\usepackage[section]{algorithm}

\usepackage{xr-hyper}
\usepackage[colorlinks,citecolor=blue,urlcolor=blue]{hyperref}

\usepackage[round]{natbib}

\newtheorem{thm}{Theorem}[section]
\newtheorem{lem}[thm]{Lemma}

\newtheorem{cor}[thm]{Corollary}

\theoremstyle{definition}
\newtheorem{defn}[thm]{Definition}
\newtheorem{exmp}[thm]{Example}
\newtheorem{rem}[thm]{Remark}

\numberwithin{equation}{section}

\newcommand\blfootnote[1]{%
  \begingroup
  \renewcommand\thefootnote{}\footnote{#1}%
  \addtocounter{footnote}{-1}%
  \endgroup
}

\newtheorem{lemma}{Lemma}[section]

\newtheorem*{assumption*}{Assumption}

\newenvironment{proofof}[1][]{\begin{trivlist}
   \item[\hskip \labelsep {\bfseries Proof of #1.}]}{\end{trivlist}}

\newcommand{\indep}{\perp\!\!\!\perp}

\DeclareMathOperator{\Var}{Var}
\DeclareMathOperator{\Exp}{E}

\newcommand{\dx}{\ \text{d}x}

\newcommand{\tod}{\overset{d}{\to}}
\newcommand{\dy}{\ \text{d}y}

\newcommand{\dF}{\ \text{d}F}

\renewcommand{\phi}{\varphi}

\newcommand{\bR}{\mathbb{R}}
\newcommand{\bN}{\mathbb{N}}

\newcommand{\ol}[1]{\overline{#1}}

\renewcommand{\hat}{\widehat}

\title{Large-Sample Theory for the Bergsma-Dassios Sign Covariance}

\author{Preetam Nandy$\dagger$}
  \address{Seminar for Statistics, ETH Z{\"u}rich, Switzerland}
  \email{nandy@stat.math.ethz.ch}
  
\author{Luca Weihs$\dagger$}
\address{Department of Statistics, University
  of Washington, Seattle, WA, U.S.A.}
\email{lucaw@uw.edu}

\author{Mathias Drton}
\address{Department of Statistics, University
  of Washington, Seattle, WA, U.S.A.}
\email{md5@uw.edu}

\begin{document}

\blfootnote{$\dagger$ Equal contribution.}

\maketitle

\begin{abstract}
  The Bergsma-Dassios sign covariance is a recently proposed extension
  of Kendall's tau.  In contrast to tau or also Spearman's rho, the
  new sign covariance $\tau^*$ vanishes if and only if the two
  considered random variables are independent.  Specifically, this
  result has been shown for continuous as well as discrete variables.
  We develop large-sample distribution theory for the empirical
  version of $\tau^*$.  In particular, we use theory for degenerate
  U-statistics to derive asymptotic null distributions under
  independence and demonstrate in simulations that the limiting
  distributions give useful approximations.
\end{abstract}


\section{Introduction}\label{section: introduction}

Many popular measures of pairwise dependence, for example Kendall's
tau \citep{Kendall38} and Spearman's rho \citep{Spearman04}, have the
undesirable property that they may be zero even when the two
considered random variables $X$ and $Y$ are dependent.  Addressing
this weakness, \cite{BergsmaDassios14} introduced a new rank-based
correlation measure $\tau^*$, which, under mild conditions on the
joint distribution of $(X,Y)$, is zero if and only if $X$ and $Y$ are
independent.  Where Kendall's tau is defined in terms of concordance
and discordance of two independent copies of $(X,Y)$, the new $\tau^*$
is based on similar notions of concordance and discordance for four
independent copies of $(X,Y)$.  While a na\"{\i}ve computation of
$t^*$, the empirical version of $\tau^*$, thus requires $O(n^4)$ time
for a sample of size $n$, it was recently shown that this
computational burden can be reduced to $O(n^2\log(n))$
\citep{WeihsEtAl15}.  As $t^*$ is now computable for larger sample
sizes, understanding its asymptotic behavior becomes a problem of
practical interest and has the potential to yield simple tests of
independence that avoid Monte Carlo approximation of p-values.

We introduce the statistic $t^*$ in Section~\ref{section:
  preliminaries}, where we also review background on U-statistics.  In
Section~\ref{sec:degeneracy}, we clarify that $t^*$ is a degenerate
U-statistic under the null hypothesis that the sample is generated
under independence.  We also prove that in certain settings degeneracy
occurs only under independence.  In Section \ref{section: asymptotics
  under the null}, we use the asymptotic theory of degenerate
U-statistics to derive an explicit representation of the asymptotic
distribution of $t^*$ when the sample is generated under independence
and with marginals that are continuous or discrete.  The asymptotic
distribution takes the form of a Gaussian chaos; specifically, we find
a (in some cases infinite) sum of scaled and centered chi-square
distributions.  Simulations in Section \ref{section: simulations} then
demonstrate how the large-sample theory can be leveraged to perform
tests of independence and compute power.  Indeed, asymptotic
distributions are found to give accurate approximations for sample
sizes as small as $n=80$.  We end with a discussion in Section \ref{section: discussion}. 


\section{Preliminaries}\label{section: preliminaries}

\subsection{The $t^*$ statistic} \label{subsection: t*} Let
$(x_1,y_1),\ldots,(x_n,y_n)$ be a sample of points in $\mathbb{R}^2$.
The empirical version of the Bergsma-Dassios sign covariance is
the statistic
\begin{align}\label{equation: test statistics}
t^* := \frac{(n-4)!}{n!} \sum_{\substack{1\leq i,j,k,l \leq n \\
    i,j,k,l~\text{distinct}}} 
a(x_i,x_j,x_k,x_l) a(y_i,y_j,y_k,y_l),
\end{align}
where 
\begin{multline}
  \label{eq:def-a}
a(z_1,z_2,z_3,z_4) =\\ I(z_1,z_3<z_2,z_4) + I(z_1,z_3>z_2,z_4) -
I(z_1,z_2<z_3,z_4) - I(z_1,z_2>z_3,z_4).
\end{multline}
Here we use $I(\cdot)$ to denote the indicator function and $a,b < c,d$ is shorthand for $\max(a,b) < \min(c,d)$.  As in
\cite{WeihsEtAl15}, we defined $t^*$ in the form of a U-statistic,
whereas \cite{BergsmaDassios14} introduced it as a V-statistic.
Indeed, $t^*$ from~(\ref{equation: test statistics}) is an unbiased
estimator of the sign covariance
\[
\tau^* :=
\Exp\left[a(X_1,X_2,X_3,X_4)a(Y_1,Y_2,Y_3,Y_4)\right]
\]
of \cite{BergsmaDassios14}.  Here, $(X_1,Y_1),...,(X_4,Y_4)$ are
random vectors drawn independently from a given bivariate distribution on
$\mathbb{R}^2$.

\begin{exmp}
  \label{ex:gaussian-taustar}
  Figure~\ref{fig:normal-taustar} shows the values of $\tau^*$ for
  bivariate normal distributions, which we computed by Monte Carlo
  simulation.  The sign covariance $\tau^*$ is an even function of the
  normal correlation $\rho$, and we thus only show values for
  $\rho\in[0,1]$.  For each considered correlation $\rho$ we
  averaged 200 values of $t^*$, each computed from a
  sample of size $n=300$.
\end{exmp}

\begin{figure}[t]
  \centering
  \includegraphics[scale=.4]{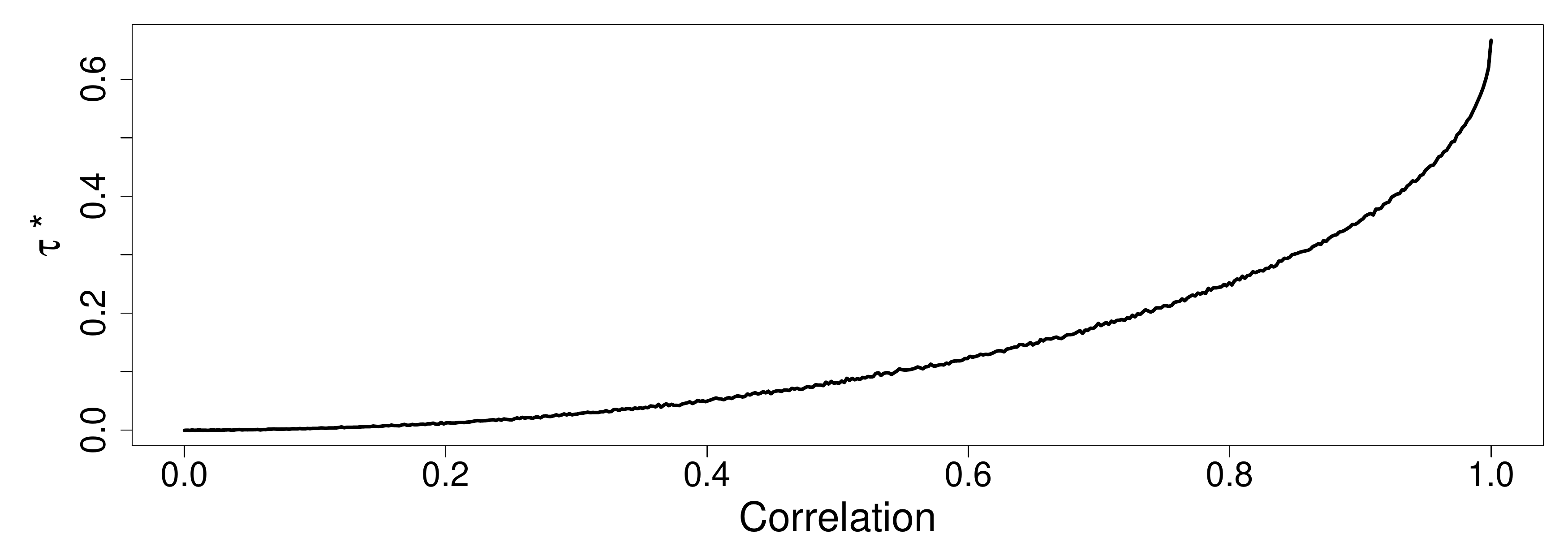}
  \caption{Sign covariance $\tau^*$ of bivariate normal
    distributions.}
  \label{fig:normal-taustar}
\end{figure}

As we explain in the remainder of this subsection, the statistic $t^*$
is based on counting concordant and disconcordant quadruples.

\begin{defn}
Let $(x_1,y_1),...,(x_4,y_4)$ be four points relabelled so that $x_1\leq x_2\leq x_3\leq x_4$.  We say that the points are 
\begin{align*}
	\begin{array}{lll}
		\text{\emph{inseparable}} & \mbox{if } \ \ 
			\begin{array}{@{}l@{}}
				\text{$x_2=x_3$ or there exists a permutation $\pi$ of $\{1,2,3,4\}$} \\
				\text{so that $y_{\pi(1)}\leq y_{\pi(2)}=y_{\pi(3)}\leq y_{\pi(4)}$,}
			\end{array}
	\end{array}
\end{align*}
and if they are not inseparable, then we call them
\begin{align*}
		\begin{array}{lll}
			\text{\emph{concordant}} & \mbox{if } \text{$\max(y_1,y_2) < \min(y_3,y_4)$\ \ or \ $\max(y_3,y_4) < \min(y_1,y_2)$,} \\
			\text{\emph{discordant}} & \mbox{if }  \text{$\max(y_1,y_2) > \min(y_3,y_4)$ and $\max(y_3,y_4) > \min(y_1,y_2) $.}
		\end{array}
\end{align*}
\end{defn}
The above definitions are mutually exclusive and exhaustive in that
any set of four points in $\bR^2$ will be exactly one of inseparable,
concordant, or discordant. Moreover, if the points are drawn from a
bivariate distribution with continuous marginals then they will be
almost surely concordant or discordant. See Figure 3 of
\cite{BergsmaDassios14} for a visual depiction of concordance and
discordance.

Let $S_4$ be the set of permutations on 4 elements, and for $\pi \in
S_4$ and $(z_1,z_2,z_3,z_4) \in \mathbb{R}^{4}$, let $z_{\pi(1,2,3,4)}
:= (z_{\pi(1)},z_{\pi(2)},z_{\pi(3)},z_{\pi(4)})$.  Introducing the
symmetric function
\begin{align}\label{eq: kernel}
h((x_1,y_1),\ldots,(x_4,y_4)) := \frac{1}{4!} \sum_{\pi \in S_4}
a(x_{\pi(1,2,3,4)}) a(y_{\pi(1,2,3,4)}),
\end{align}
we may rewrite $t^*$ as a sum of permutation invariant terms,
namely, 
\begin{align}\label{eq: Bergsma-Dassios U-statistic}
t^* = \frac{1}{\binom n4} \sum_{(i,j,k,l) \in C(n,4)} h\left((x_i,y_i),(x_j,y_j),(x_k,y_k),(x_l,y_l)\right),
\end{align}
where $C(n,4)= \{(i,j,k,l) \,:\,1 \leq i < j < k < l \leq n \}$.  Lemma
1 in \cite{WeihsEtAl15} gives the following result.

\begin{lem} \label{lemma: kernel values}
Let $A = \{(x_1,y_1),(x_2,y_2),(x_3,y_3),(x_4,y_4)\}\subset\mathbb{R}^2$. Then
\begin{align*}
	h((x_1,y_1),..,(x_4,y_4)) = \left\{
		\begin{array}{lll}
			2/3 & \mbox{if } \text{the points in $A$ are concordant,}\\
			-1/3 & \mbox{if }  \text{the points in $A$ are discordant,} \\ 
			0 & \mbox{if }  \text{the points in $A$ are inseparable.}
		\end{array}
	\right.
\end{align*}
\end{lem}

Equation~(\ref{eq: Bergsma-Dassios U-statistic}) expresses $t^*$ in
the familiar form of a U-statistic with symmetric kernel $h$, and we
proceed to review some of the tools available for the study of
U-statistics.

\subsection{Theory of U-statistics} \label{subsection: theory of
  u-statistics} Let $Z_1,Z_2,\dots$ be i.i.d.\ random variables taking
their values in $\bR^d$ with $d\geq 1$.  Let $k:\left(\bR^d\right)^m
\to \bR$ be a \emph{kernel function} invariant to permutation of its
$m$ 
arguments. For $n\geq m$, the \emph{U-statistic with kernel $k$} is
the statistic
\begin{align}
  \label{eq:Un}
	U_n := \frac{1}{{n \choose m}} \sum_{(i_1,...,i_m)\in C(n,m)} k(Z_{i_1},...,Z_{i_m}),
\end{align}
where $C(n,m) = \{(i_1,...,i_m)\in \{1,...,n\}^m: i_1 < i_2 < ... <
i_m\}$.  Note that $\Exp[U_n] = \Exp[k(Z_1,...,Z_m)]$ so that $U_n$ is
an unbiased estimator of $\theta:=\Exp[k(Z_1,...,Z_m)]$.

Of central importance in determining the asymptotics of U-statistics
are the functions 
\begin{equation}
  \label{eq:k_i}
  k_i(z_1,...,z_i) =  \Exp[k(z_1,...,z_i, Z_{i+1},\dots,Z_{m})], \quad i
=1,\dots,m,
\end{equation}
 and their variances
\begin{equation}
  \label{eq:sigma_i^2}
  \sigma_i^2 = \Var[k_i(Z_1,...,Z_i)], \quad i=1,...,m.
\end{equation}
It is well known that $\sigma_1^2 \leq \sigma_2^2 \leq ...\leq
\sigma_m^2$.  In particular, if $\sigma_m^2$ is finite then so are all
other $\sigma_i^2$.  We now recall two theorems on the large-sample
distribution of the U-statistic $U_n$ \cite[Chapter 5]{Serfling:1980}.

\begin{thm} \label{theorem: normal asymptotics} 
  If the kernel $k$ of the statistic $U_n$ from~(\ref{eq:Un})
  has variance $\sigma_m^2 < \infty$, then
\begin{align*}
	\sqrt{n}(U_n-\theta) \tod N(0, m^2 \sigma_1^2).
\end{align*}
\end{thm}

If $\sigma_1^2 = 0$, then the Gaussian limit is degenerate, and we
have $\sqrt{n}(U_n-\theta) \overset{p}{\to} 0$.  Indeed, if
$\sigma_1^2 = 0$ and $\sigma_2^2 >0$, then scaling $U_n$ by a factor
of $n$ results in a non-Gaussian asymptotic distribution.  To present
this result, we write $\chi^2_1$ for the chi-square distribution with
one degree of freedom and define $A_k$ to be the operator that acts
via $g(\cdot) \mapsto \Exp[(k_2(\cdot, Z_1) - \theta)g(Z_1)]$  on
square-integrable functions $g$ (that is,  $\Exp[g(Z_1)^2]<\infty$).

\begin{thm} \label{theorem: degenerate asymptotics} If the kernel $k$
  of the statistic $U_n$ from~(\ref{eq:Un}) has variance
  $\sigma_m^2 < \infty$ and $\sigma_1^2 = 0$, then
\begin{align*}
	n(U_n - \theta) \tod {m \choose 2}\sum_{i=1}^\infty \lambda_i (\chi_{1i}^2 - 1)
\end{align*}
where $\chi_{11}^2,\chi_{12}^2,\dots$ are i.i.d.~$\chi^2_1$ random
variables, and the $\lambda_i$'s are the eigenvalues, taken with
multiplicity, associated to a system of orthonormal eigenfunctions of
  the operator $A_k$.
\end{thm}

We will use Theorem \ref{theorem: degenerate asymptotics} in Section
\ref{section: asymptotics under the null} to find the asymptotic
distribution of $t^*$ under the null hypothesis of independence. 



\section{Degeneracy of the sign covariance}
\label{sec:degeneracy}

Let $Z_i=(X_i,Y_i)$ for $i = 1,2,\dots,n$ be an i.i.d.\ sequence
comprising copies of a random vector $(X,Y)$ with values in $\bR^2$.
Let $t^*$ be the (empirical) Bergsma-Dassios sign covariance for this
sample.  We begin our study of the asymptotic properties of the
U-statistic $t^*$ by studying its degeneracy.  Our first observation
is that $t^*$ is degenerate when $X$ and $Y$ are independent, denoted
$X\indep Y$.  Next, in a particular setting that has $(X,Y)$
continuously distributed, we are able to show that $t^*$ is degenerate
only if $X\indep Y$.

The statistic $t^*$ has the kernel $h$ from \eqref{eq: kernel}, which
has $m=4$ arguments.  Specializing the definitions from~(\ref{eq:k_i})
and~(\ref{eq:sigma_i^2}) to the present setting, we may define
functions $h_1,...,h_4$ with variances $\sigma_1^2,...,\sigma_4^2$.
The kernel $h$ is a bounded function and thus $\sigma_4^2<\infty$.
Hence, Theorem~\ref{theorem: normal asymptotics} applies and yields
the following result.

\begin{cor}\label{cor: asymptotics under dependence}
  As $n\to\infty$, the sign covariance converges to a normal limit,
  namely, 
  \begin{align*}
    \sqrt{n}(t^* - \tau^*) \tod N(0, 16\sigma_1^2).
  \end{align*}
\end{cor}

The result just stated provides a non-trivial distributional
approximation to $t^*$ only if $\sigma_1^2>0$.  The following lemma
observes that this fails to be the case under the null hypothesis of
independence, under which $t^*$ is a degenerate U-statistic.  The
proof of the lemma as well as the proofs of all other results in this
section are deferred to Appendix~\ref{section: proofs-degen}.

\begin{lem}\label{lemma: degeneracy of h_1}
  If $X\indep Y$ then $\sigma_1^2 = \Var[h_1(X_1,Y_1)] =0$ so that
  $h_1(X_1,Y_1)$ is a degenerate random variable.
\end{lem}

According to Lemma~\ref{lemma: degeneracy of h_1} and Theorem
\ref{theorem: normal asymptotics}, if $X\indep Y$ we have
$\sqrt{n}t^*\overset{p}{\to} 0$ because $\Exp[t^*] =\tau^*= 0$ under
independence \citep{BergsmaDassios14}.  We thus need to appeal to
Theorem \ref{theorem: degenerate asymptotics} to find a non-degenerate
asymptotic distribution for $t^*$ when $X\indep Y$.  This is the topic
of Section~\ref{section: asymptotics under the null}.

\begin{rem} \label{remark: sigma values} In the continuous case with
  $X\indep Y$, it is possible to compute all of the variances
  $\sigma_1^2,...,\sigma_4^2$ exactly.  We report these values to be
  \begin{align*}
    \sigma_1^2 = 0,\ \ \ \sigma_2^2 = \frac{1}{225}, \ \ \ \sigma_3^2 = \frac{8}{225}, \ \ \ \sigma_4^2 = \frac{50}{225}.
  \end{align*}
  The fact that $\sigma_1^2=0$ was shown in generality in
  Lemma~\ref{lemma: degeneracy of h_1}.  The value of $\sigma_2^2$ can
  be computed as the sum of the squared eigenvalues of $h_2$ which are
  derived in the proof of Theorem \ref{theorem: continuous asymptotic
    distribution}; in particular, we have that
  $\sigma_2^2 = \sum_{i=1}^\infty\sum_{j=1}^\infty \frac{6^2}{\pi^8
    i^4j^4} = \frac{1}{225}$. Finally, $\sigma_3^2$ and $\sigma_4^2$
  can be computed explicitly from the representation of $h$ in Lemma
  \ref{lemma: kernel values}, this computation is trivial for
  $\sigma_4^2$ but quite lengthy for $\sigma_3^2$ and thus is omitted.
\end{rem}

Next, we turn our attention to the case that $X\not\indep Y$ and
$(X,Y)$ are generated from a continuous distribution on
$\bR^2$.  In this case we find $t^*$ to be non-degenerate.

\begin{thm}\label{theorem: nondegeneracy of h_1}
  Suppose $(X,Y)$ has a bivariate continuous distribution with a
  continuous density function $f$ with support
  $\ol{f^{-1}((0, \infty))} = [a,b]\times [c,d]$, where
  $-\infty\le a<b\le \infty$ and $-\infty\le c<d\le \infty$.  If $X$
  and $Y$ are dependent, then $\sigma_1^2 = \Var[h_1(X_1,Y_1)] > 0$.
\end{thm}

In the setting of Theorem~\ref{theorem: nondegeneracy of h_1}, we thus
have that $t^* = N(\tau^*, 16 \sigma_1^2/n) + o_p(n^{-1/2})$. 

\begin{exmp} \label{example: variance of h1} To gain intuition for the
  magnitude of the asymptotic variance $16\sigma_1^2$, we use Monte
  Carlo integration to compute $16\sigma_1^2$ in the case that $(X,Y)$
  follow a bivariate normal distribution.  Since
  $\sigma_1^2$ is an even function of the correlation $\rho$ of a
  bivariate normal distribution, we consider $\rho \in [0,1]$.  In
  particular, we perform this computation letting $\rho$ take on 20,
  evenly spaced, values between $0$ and $1$.  The results of this
  computation are shown in Figure \ref{figure: variance of h1}. The
  figure shows that the asymptotic variance $16 \sigma_1^2$ gradually
  increases with larger values of $\rho$, peaking at $16 \sigma_1^2
  \approx 0.14$ when $\rho \approx 0.74$, and then decreases to 0 as
  the correlation further approaches 1. Note that a value of
  $\sigma_1^2 = 0$ when $\rho=1$ does not contradict
  Theorem~\ref{theorem: nondegeneracy of h_1} as, in this case, the
  joint distribution of $(X,Y)$ is not continuous. The shape of the curve in Figure
  \ref{figure: variance of h1} can be partially explained by the fact that
  $\sigma_1^2 \leq \sigma_4^2 = \Var(h(Z_1,...,Z_4)) = (\tau^* + 1/3)(1 - (\tau^* + 1/3))$.
  For instance, note that $(\tau^* + 1/3)(1 - (\tau^* + 1/3))$ equals 0 when $\tau^* = 2/3$
  (in which case the correlation can be seen to be 1 or -1), and is maximized at
  $\tau^* = 1/6 \approx .167$ which corresponds a correlation of approximately 0.7 (see
  Figure \ref{fig:normal-taustar}).
\end{exmp}

\begin{figure}
        \centering
        \includegraphics[width=.75\textwidth]{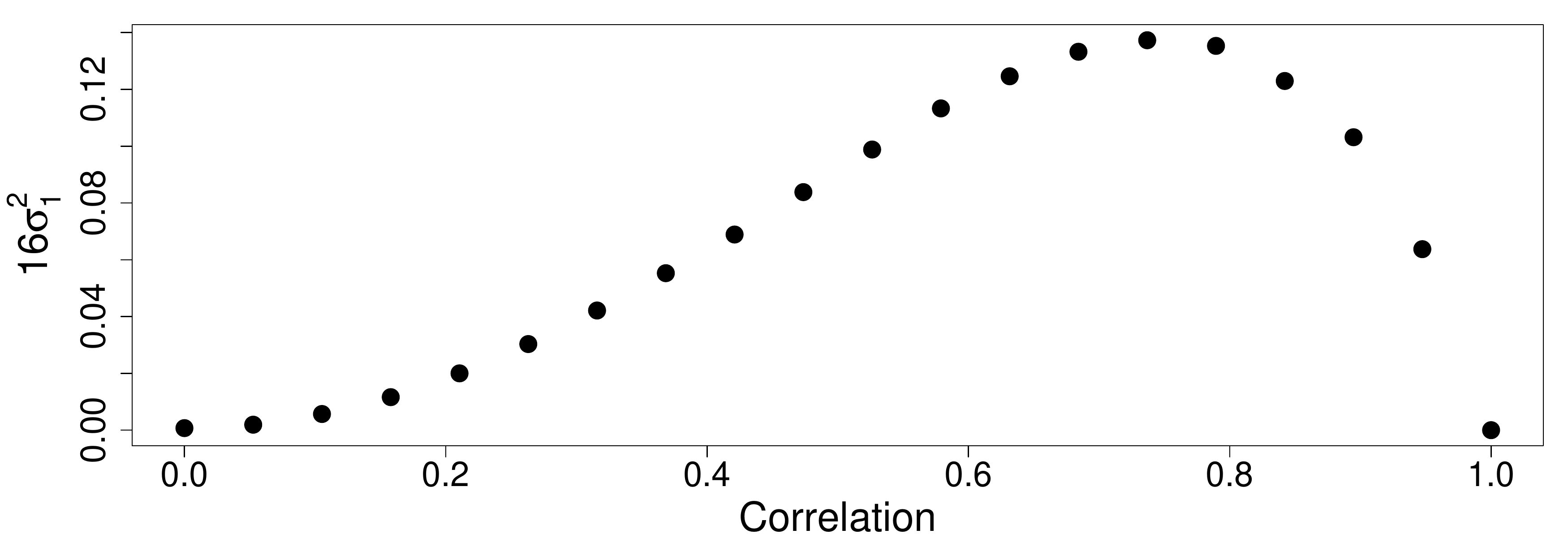}
        \caption{Monte Carlo approximations to the values of
          $16\sigma_1^2$ for bivariate normal distribution with
          different correlations.}
        \label{figure: variance of h1}
\end{figure}

\section{Asymptotics under the null hypothesis of
  independence} \label{section: asymptotics under the null} 

As in the previous section, let $t^*$ be the empirical sign covariance
for an i.i.d.~sample $(X_i,Y_i)$, $i = 1,2,\dots,n$, with values in
$\bR^2$.  Throughout this section, we assume the $(X_i,Y_i)$ to be
independent copies of a random vector $(X,Y)$ with $X\indep Y$, so
that $t^*$ is degenerate (Lemma~\ref{lemma: degeneracy of h_1}).  We
thus need to appeal to Theorem \ref{theorem: degenerate asymptotics}
to find a non-degenerate asymptotic distribution for $t^*$.  Since
$\Exp[t^*] =\tau^*= 0$ under independence, we are led to the problem
of determining the eigenvalues of the operator
$A_h:g(\cdot) \mapsto \Exp[h_2(\cdot, Z_1)g(Z_1)]$.  

A key observation is that under independence $A_h$ is a tensor product
of operators because the function $h_2$ admits the following
factorization, which along with all other results in this section is
proved in Appendix~\ref{section: proofs}.

\begin{lem}\label{lemma: product formula for h_2}
  If $X\indep Y$ then
  \[
    h_2((x_1,y_1),(x_2,y_2)) = \frac{2}{3}\,g_{X}(x_1,x_2)
    \,g_{Y}(y_1,y_2),
  \] 
  where $g_{X}(x_1,x_2) = \Exp[a(x_1,x_2,X_3,X_4)]$
  and $g_{Y}(y_1,y_2) = \Exp[a(y_1,y_2,Y_3,Y_4)]$.
\end{lem}

The function $g_X$ (and similarly $g_Y$) takes the form
\begin{align}
  \label{eq:prob-form-of-g}
	 g_{X}(x_1,x_2) 
 	&\ = P(x_1, X_3 < x_2, X_4) + P(x_1, X_3 > x_2, X_4) \\
  \nonumber
	&\hspace{6mm} - P( x_1, x_2 < X_3, X_4) - P(x_1, x_2 > X_3, X_4).
\end{align}

By Lemma~\ref{lemma: product formula for h_2}, $A_h=A_{g_X}\otimes
A_{g_Y}$ and thus the spectrum of $A_h$ is the product of the spectra
of $A_{g_X}$ and $A_{g_Y}$.  We record the general version of this
fact in the next lemma.  Here, eigenvalues are always repeated
according to their multiplicity, and we let
$\bN_+=\{1,2,\dots\}$.

\begin{lem} \label{lemma: h_values} Let $g_1$ and $g_2$ be symmetric
  real-valued functions with $\Exp[g_1(X_1,X_2)] = \Exp [g_2(Y_1,Y_2)]
  = 0$ and $\Exp[g_1(X_1,X_2)^2], \Exp [g_2(Y_1,Y_2)^2] <\infty$.  For
  $i=1,2$, let $\lambda_{i,j}$, $j\in\bN_+$, be the nonzero eigenvalues
  of $A_{g_i}$.
  Then the products $\lambda_{1,j_1}\lambda_{2,j_2}$, $j_1,j_2\in\bN_+$,
  are the nonzero eigenvalues of $A_k$ for $k((x_1,y_1),(x_2,y_2)) :=
  g_1(x_1,x_2)g_2(y_1,y_2)$.
\end{lem}

In the sequel, we use the factorization results from
Lemmas~\ref{lemma: product formula for h_2} and~\ref{lemma: h_values}
to obtain the asymptotic distribution of $t^*$ when $X$ and $Y$ are
continuous (Section~\ref{subsection: continuos}), and when $X$ and $Y$
are discrete with finite support (Section~\ref{subsection: discrete}).
A straightforward extension covers the mixed continuous and discrete
case (Section~\ref{subsection: discrete}).

\subsection{Continuous variables}\label{subsection: continuos}

Suppose now that $X\indep Y$ with $X$ and $Y$ following continuous
marginal distributions.  Since $h((X_1,Y_1),...,(X_4,Y_4))$ depends
only on the joint ranks of $(X_1,Y_1),...,(X_4,Y_4)$, it follows that
$\tau^*$ (and $t^*$) are invariant to monotonically increasing 
transformations of the marginals of $(X,Y)$.  As such we may, and
will, assume that $X$ and $Y$ are i.i.d.~$\text{Uniform}(0,1)$.  Then
$(X,Y)$ is uniform on the unit square $(0,1)\times (0,1)$.  In this
case the factorization described in Lemma \ref{lemma: product formula
  for h_2} has a particularly nice form.

\begin{lem} \label{lemma: continuous factorization} 
  If $X,Y\overset{i.i.d.}{\sim}\text{Uniform}(0,1)$, then for $(x_1,y_1),(x_2,y_2) \in (0,1)^2$,
  \begin{align*}
    h_2((x_1, y_1), (x_2, y_2)) &= 6\, c(x_1,x_2)\,c(y_1,y_2)
\end{align*}
where
\[
  c(x_1,x_2)=\frac{1}{2} x_1^2+ \frac{1}{2}x_2^2- x_1\vee
  x_2+\frac{1}{3}
\]
and $x_1\vee x_2:=\max\{x_1,x_2\}$.
\end{lem}

Somewhat surprisingly, the function $c$ corresponds to the kernel of
the well studied Cram\'{e}r-von Mises statistic.  Leveraging the fact
that the eigenvalues of $A_c$ are already known, we are now able to
derive the asymptotic distribution of $t^*$.

\begin{thm} \label{theorem: continuous asymptotic distribution} If $X$
  and $Y$ are independent continuous random variables, then
\begin{align*}
	nt^* &\overset{d}{\to} \frac{36}{\pi^4}\sum_{i=1}^\infty\sum_{j=1}^\infty \frac{1}{i^2j^2}(\chi_{1,ij}^2-1)
\end{align*}
where $\{\chi^2_{1,ij} : i,j \in \bN_+ \}$ is a collection of i.i.d.\
$\chi^2_1$ random variables. 
\end{thm}

Remarkably the asymptotic distribution just given is simply a scale
multiple of the asymptotic distribution of the U-statistic for
Hoeffding's $D$ where
\[
  D=\iint (F_{X,Y}(x,y)-F_X(x)F_Y(y))^2 \dF_{X,Y}(x,y);
\]
see \cite{Hoeffding48}.  When $(X,Y)$ has a continuous joint
distribution, it is readily seen that $D=0$ if and only if
$X\indep Y$.  However, this may fail in non-continuous cases.

\subsection{Discrete variables} \label{subsection: discrete} 

We now treat the case where $X$ and $Y$ are independent discrete
random variables with finite supports.  Unlike in the continuous case,
the asymptotic distribution of $t^*$ then depends on how $X$ and $Y$
distribute their probability mass marginally. In practical applications these
marginal probabilities must be estimated before using our limit theorem.

In order to present the result, we associate a matrix to a discrete
random variable as follows.  Let $U$ be a random variable with finite
support $\{u_1,\dots,u_r\}$, cumulative distribution function $F_U$
and probability mass function $p_U$.  We then define $R^U$ 
to be the $r\times r$ symmetric matrix whose $(i,j)$-th entry is
\begin{align} \label{equation: R matrix form}
	R^U_{ij} =&
	\ \sqrt{p_U(u_i)p_U(u_j)} \Bigg\{ \bigg[(F_U(u_i\wedge u_j) - p_U(u_i\wedge u_j))^2 + (1-F_U(u_i\vee u_j))^2\bigg] \\
	&- I(u_i\not=u_j)\bigg[F_U(u_i\wedge u_j)(1-F_U(u_i\wedge u_j)) 
	+ \sum_{u_i\wedge u_j < u_\ell  < u_i\vee u_j} p_U(u_\ell)(1-F_U(u_\ell))\bigg] \nonumber
	 \Bigg\}. 
\end{align}

\begin{thm}\label{theorem: discrete asymptotic distribution}
  Let $X$ and $Y$ be independent discrete random variables with
  finite supports of size $r$ and $s$, respectively.  Let
  $\lambda_1^X,\ldots,\lambda_r^X$ be the eigenvalues of $R^X$, and
  let $\lambda^Y_1,\ldots,\lambda^Y_s$ be the eigenvalues of $R^Y$.
  Then
 \begin{align*}
	nt^* &\overset{d}{\to} 4\sum_{i=1}^r\sum_{j=1}^s \lambda_i^X \lambda_j^Y (\chi_{1,ij}^2-1)
\end{align*}
where $\{\chi^2_{1,ij} : i \leq r,~j\leq s \}$ is a collection of $rs$ i.i.d.\ $\chi_1^2$ random variables.
\end{thm}

In the special case that $X$ and $Y$ are Bernoulli random variables,
the asymptotic distribution can be presented in simple form.

\begin{exmp}
  If $X\sim\text{Bernoulli}(p)$ for $p\in(0,1)$, then 
  \begin{align*}
    R^X = \begin{pmatrix}
      p^2(1-p) & -(p(1-p))^{3/2} \\
      -(p(1-p))^{3/2} & p(1-p)^2 
    \end{pmatrix}
  \end{align*}
  has rank one and its nonzero eigenvalue is $p(1-p)$.  It follows
  that if $Y$ is a second independent random variable with
  $Y\sim\text{Bernoulli}(q)$ for $q\in(0,1)$, then
  \[
    nt^* \overset{d}{\to} 4pq(1-p)(1-q)(\chi^2_1 - 1).
  \]
  So, $t^*$ can be centered and scaled to become asymptotically
  chi-square.
\end{exmp}

\begin{exmp}
  For a ternary random variable $X$ with $P(X=1)=p_1$, $P(X=2)=p_2$
  and $P(X=3)=p_3=1-p_1-p_2$, we have
  \begin{align*}
    R^X =  \begin{pmatrix}
      p_1(1-p_1)^2 
      & -\sqrt{p_1p_2}\left[ p_1(1-p_1)-p_3^2\right]
      & -\sqrt{p_1p_3}\left[ p_3(1-p_3)+p_1p_2\right]\\
      .
      & p_2\left( p_1^2+p_3^2\right)
      & -\sqrt{p_2p_3}\left[p_3 (1 -p_3)- p_1^2 \right]\\
      .
      & .
      &p_3(1-p_3)^2\\
    \end{pmatrix},
  \end{align*}
  where we show only the upper half of the symmetric matrix.  No
  simple formula seems to be available to determine the eigenvalues of
  $R^X$ in this case, but the eigenvalues can readily be computed
  numerically for any (possibly estimated) values of $p_1$ and $p_2$.
\end{exmp}

Finally,  if $X$ is discrete with finite support and $Y$ is
continuous, then a simple extension of Theorems \ref{theorem:
  continuous asymptotic distribution} and \ref{theorem: discrete
  asymptotic distribution} gives the following result.

\begin{cor}\label{theorem: mixed asymptotic distribution}
  Let $X$ and $Y$ be independent random variables, where $X$ has
  finite support of size $r$ and $Y$ is continuous. Let $\lambda_1,...,\lambda_r$
  be the eigenvalues of $R^X$. Then
 \begin{align*}
   nt^* &\overset{d}{\to} \frac{12}{\pi^2}\sum_{i=1}^r\sum_{j=1}^\infty \frac{\lambda_i}{j^2} (\chi_{1,ij}^2-1)
\end{align*}
where $\{\chi^2_{1,ij} : i \leq r,~j \in\bN_+ \}$ is a collection of
i.i.d.\ $\chi^2_1$ random variables.
\end{cor}


\section{Simulations}\label{section: simulations}

The results from Section~\ref{section: asymptotics under the null} can
be used to form asymptotic tests of independence, and we now explore
which sample sizes are needed for the asymptotic approximations to be
accurate.  As a test based on $t^*$ has asymptotic power against all
alternatives to independence, it is also of interest to make
comparisons against other tests known to be (most) powerful for
particular settings and alternatives.  Finally, we demonstrate how the
results of Section~\ref{sec:degeneracy} can be used for sample size
computations.  Code for performing asymptotic tests of independence
has been 
incorporated in the TauStar\footnote{See
 \url{https://cran.r-project.org/web/packages/TauStar/index.html}} R
package available on CRAN, the Comprehensive R Archive Network
\citep{R, taustar-package}.


\subsection{Empirical convergence to the asymptotic
  distribution} \label{subsection: empirical convergence} 

Let $t^*$ be computed from a sample of size $n$ drawn from the joint
distribution of a bivariate random vector $(X,Y)$ with $X\indep Y$.
Since $t^*$ only depends on ranks, its distribution does not change
when applying monotonically increasing marginal transformations to $X$
and $Y$.  When $X$ and $Y$ both have continuous distributions, we may
thus transform their distributions to $N(0,1)$ without changing the
distribution of $t^*$.  When one or both of $X$ and $Y$ are discrete
however, the distribution of $t^*$ depends on how $X$ and $Y$
distribute their probability mass making it impossible to provide an
exhaustive empirical study of convergence properties.  Instead we will
consider selected examples.  Specifically, we consider the following cases:
\begin{enumerate}[(i)]
	\item The continuous case with $X,Y\sim N(0,1)$.
	\item A discrete case with $P(X=i) = 1/10$ for $1\leq i\leq
          10$, and $P(Y=i) \propto 2^{-i}$ for $1\leq i\leq 12$.
	\item A mixed case with $X\sim N(0,1)$ and $P(Y = i) = 1/5$
          for $1\leq i\leq 5$.
\end{enumerate}
In each setting we compute, for different sample sizes $n$, a kernel
density estimate for the distribution of $t^*$ and plot it alongside
the asymptotic density.  The resulting plots are shown in Figure
\ref{figure: density convergence}, which demonstrates that the
asymptotic and finite-sample distributions are in close agreement
already when $n=80$.  While we present only one example each for the
discrete and mixed cases we found similar results when simulating with
many other choices of distributions.

\begin{figure}
        \centering
         \subfloat[Continuous case\label{figure: continuous density}]{%
		\includegraphics[width=.33\textwidth]{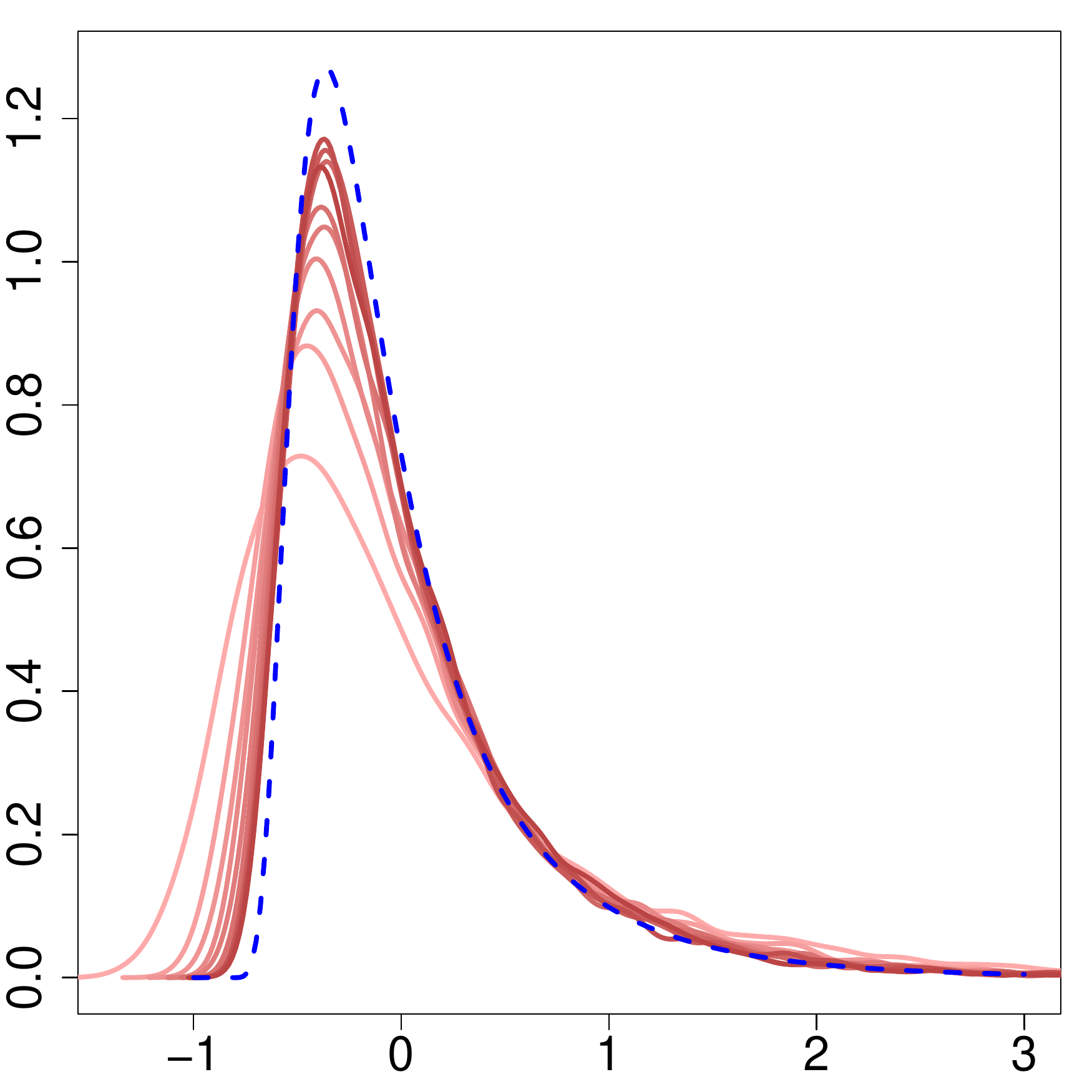}
	}
	\subfloat[Discrete case\label{figure: discrete density}]{%
		\includegraphics[width=.33\textwidth]{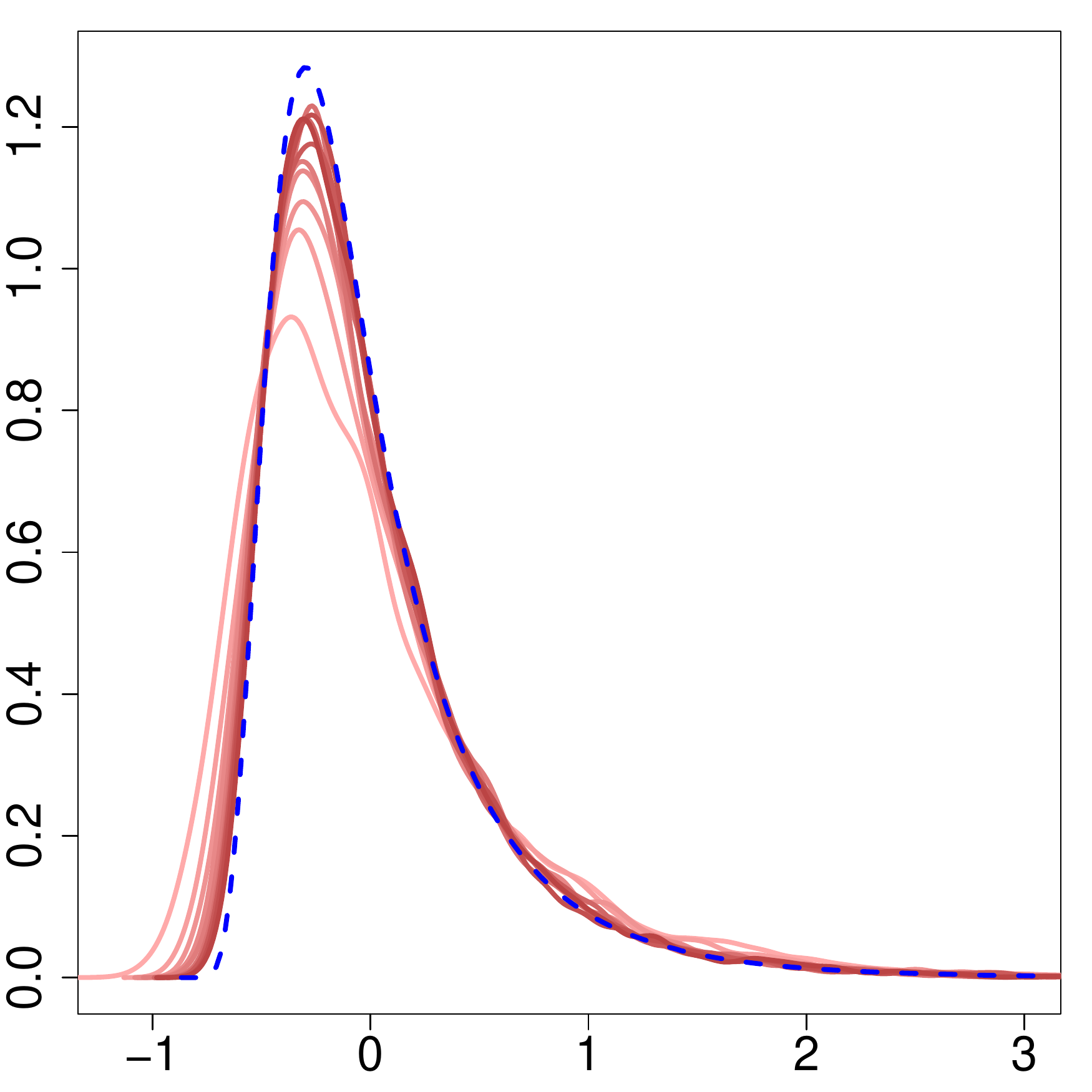}
	}
	\subfloat[Mixed case\label{figure: mixed density}]{%
		\includegraphics[width=.33\textwidth]{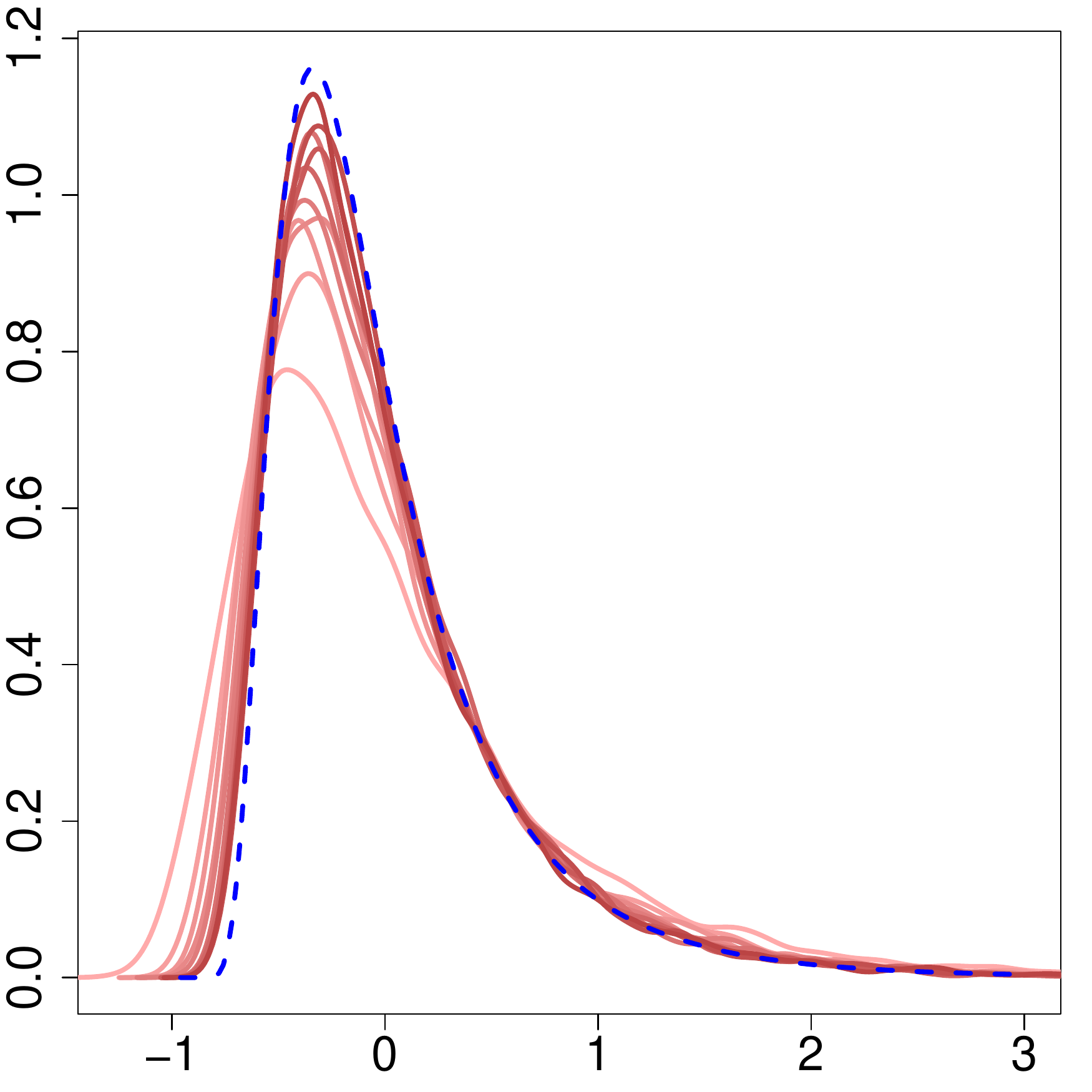}
	} 
        \caption{Kernel density estimates from 10,000 simulated values
          of $nt^*$ at sample sizes
          $n \in \{10, 15, 20, 25, 30, 40, 50, 60, 70, 80\}$; smaller
          values of $n$ are shown in lighter color.  The plots also
          show the density of the asymptotic distributions from
          Section \ref{section: asymptotics under the null} in dashed
          blue line.}
        \label{figure: density convergence}
\end{figure}

\begin{rem}
  Computing the asymptotic densities shown in Figure \ref{figure:
    density convergence} is non-trivial and requires the numerical
  inversion of the characteristic function for the asymptotic
  distributions.  To perform this numerical inversion we use the
  techniques described in Section 7 of \citet*{BlumEtAl61}; these
  computations are done automatically in the aforementioned TauStar
  package for R.
\end{rem}

\subsection{Power comparisons} \label{subsection: power comparisons}

We explore the power of an asymptotic test based on $t^*$ in six
cases:
\begin{enumerate}[(i)]
\item First, we take $(X,Y)$ as bivariate normal with correlation
  $\rho \in\{0, .1, .2, \dots, 1\}$; the distribution of $t^*$ then does
  not depend on the means and variances which may thus be set to zero
  and one, respectively.  We compare the test based on $t^*$ to the
  two-sided test based on the standard Pearson 
  correlation $\hat{\rho}$.  We implement the latter test using the
  fact that $\hat{\rho}\sqrt{(n-2)/(1-\hat{\rho}^2)}$ has a
  $t$-distribution with $n-2$ degrees of freedom.

\item Next, we consider three discrete cases all of which have $(X,Y)$
  taking values in the grid $\{1,2,\dots,5\}^2$.  In each of these cases
  we compare our test to the chi-square test of independence.
  \begin{enumerate}[(a)]
  \item In the first discrete case, $(X,Y)$ follows a mixture between
    the uniform distribution on $\{(1,1),(2,2),\dots,(5,5)\}$ and the
    uniform distribution on $\{1,\dots,5\}^2$, illustrated in Figures
    \ref{figure: dg ordinal} and \ref{figure: dg uniform},
    respectively.  We let the mixture weight $p$ for the former
    component range through the set $\{0,.1,\dots,1\}$.  
  \item The second discrete case is analogous but a mixture between
    the distributions from Figures \ref{figure: dg nonordinal} and
    \ref{figure: dg uniform}.
  \item The third discrete case is as the previous two but mixes the
    distributions from Figures \ref{figure: dg L shape} and
    \ref{figure: dg uniform}.
  \end{enumerate}
	
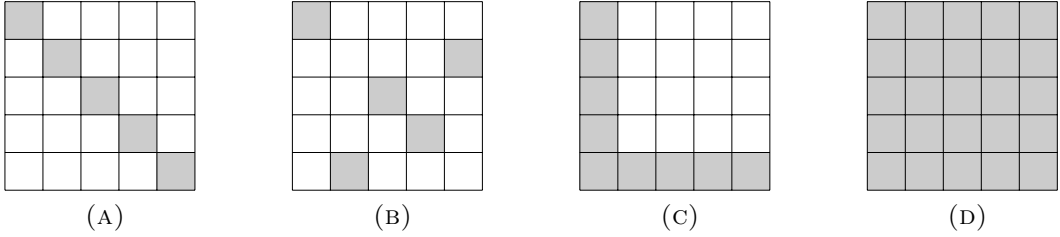
\begin{figure}
        \centering
         \subfloat[\label{figure: dg ordinal}]{%
		\begin{tikzpicture}
			\fill[black!20!white] (0,2) rectangle (.5,2.5);
			\fill[black!20!white] (.5,1.5) rectangle (1,2);
			\fill[black!20!white] (1,1) rectangle (1.5,1.5);
			\fill[black!20!white] (1.5,.5) rectangle (2,1);
			\fill[black!20!white] (2,0) rectangle (2.5,.5);
			\draw[step=0.5cm,black,very thin] (0,0) grid (2.5,2.5);
		\end{tikzpicture}
	}
	\hspace{10mm}
	\subfloat[\label{figure: dg nonordinal}]{%
		\begin{tikzpicture}
			\fill[black!20!white] (0,2) rectangle (.5,2.5);
			\fill[black!20!white] (2,1.5) rectangle (2.5,2);
			\fill[black!20!white] (1,1) rectangle (1.5,1.5);
			\fill[black!20!white] (1.5,.5) rectangle (2,1);
			\fill[black!20!white] (.5,0) rectangle (1,.5);
			\draw[step=0.5cm,black,very thin] (0,0) grid (2.5,2.5);
		\end{tikzpicture}
	}
	\hspace{10mm}
        	\subfloat[\label{figure: dg L shape}]{%
		\begin{tikzpicture}
			\fill[black!20!white] (0,0) rectangle (2.5,.5);
			\fill[black!20!white] (0,0) rectangle (.5,2.5);
			\draw[step=0.5cm,black,very thin] (0,0) grid (2.5,2.5);
		\end{tikzpicture}
	} 
	\hspace{10mm}
	\subfloat[\label{figure: dg uniform}]{%
		\begin{tikzpicture}
			\fill[black!20!white] (0,0) rectangle (2.5,2.5);
			\draw[step=0.5cm,black,very thin] (0,0) grid (2.5,2.5);
		\end{tikzpicture}
	}
        \caption{Visualization of where probability mass is placed for
          different discrete distributions on $\{1,...,5\}^2$. In each
          case the distribution is uniform over the gray squares, and
          zero probability is assigned to the white squares.} \label{figure: discrete grids}
\end{figure}

\item Finally, we experiment with two mixed cases in which the
  distribution of $X$ is
  discrete and the conditional distribution of $Y$ given $X=x$ is
  a normal distribution $N(\mu_x,1)$.
  \begin{enumerate}[(a)]
  \item The first mixed case has $X\sim\text{Bernoulli}(.3)$,
    $\mu_x=0$ when $x=1$, and $\mu_x=\mu$ when $x=0$.  Here, we let
    the mean difference $\mu$ range through the set $\{0, 1/6,
    2/6,..., 9/6\}$, and each setting we compare against the
    two-sample $t$-test.
  \item In the second case $X\sim\text{Uniform}(\{1,...,6\})$, and $Y$
    has conditional mean $\mu_x$ is zero when $x$ is odd and equal to
    $\mu \in \{0, 1/6, 2/6,..., 9/6\}$ when $x$ is even.  Here, we
    compare against a bootstrapped permutation test using the distance
    covariance statistic of \cite{SzekelyEtAl07}, known to be
    consistent for independence, using the Energy R package
    \cite{RizzoSzekely14}.
  \end{enumerate}
\end{enumerate}

The simulation results are presented in Figure \ref{figure: power
  simulations}.  Surprisingly, the $t^*$ test has competitive power in
cases (i) and (iii)(a) where the alternative tests are known to be
most powerful given the distributional assumption of normality. For
the jointly discrete cases, we observe that the chi-square test of
independence has essentially equal power in case (ii)(a),
significantly higher power in case (ii)(b), and significantly lower
power in case (ii)(c). The lack of power in case (ii)(b) is not
surprising as the $t^*$ statistic is ordinal in nature and the
dependence in the distribution from case (ii)(b) was designed to be
non-ordinal.  The ordinal nature of $t^*$ also explains the
significant gains in case (ii)(c).  Hence, it would seem that the
$t^*$ test for jointly discrete data can offer substantial
improvements in power over the chi-square test if an ordinal
dependence relationship is suspected in the data. Finally, case
(iii)(b) suggests that there are cases in which $t^*$ may provide
higher power than the distance covariance.

\begin{figure}
        \centering
         \subfloat[Case (i)\label{figure: case (i)}]{%
         		\includegraphics[width=.35\textwidth]{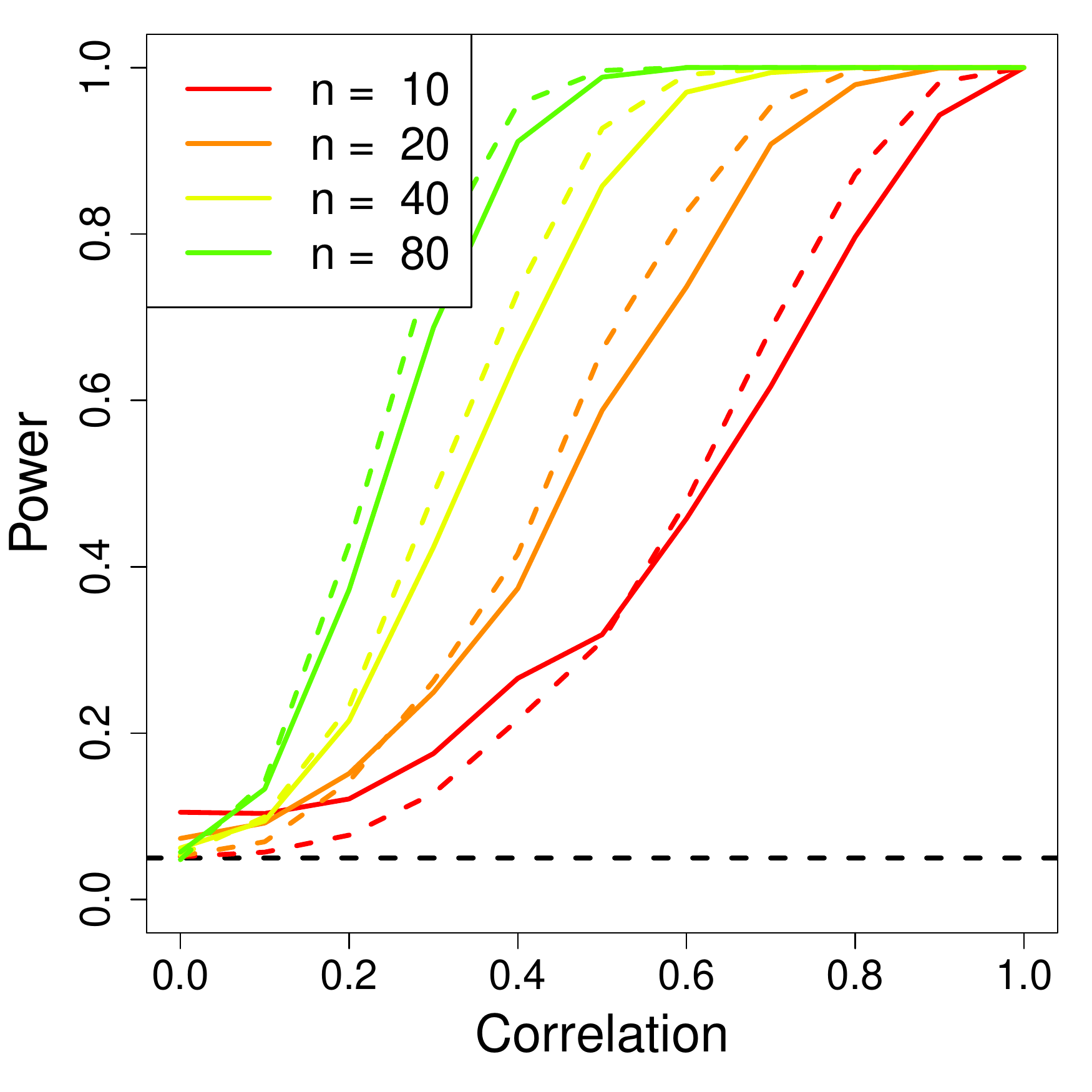}
	}
	\subfloat[Case (ii)(a)\label{figure: case (ii)(a)}]{%
         		\includegraphics[width=.35\textwidth]{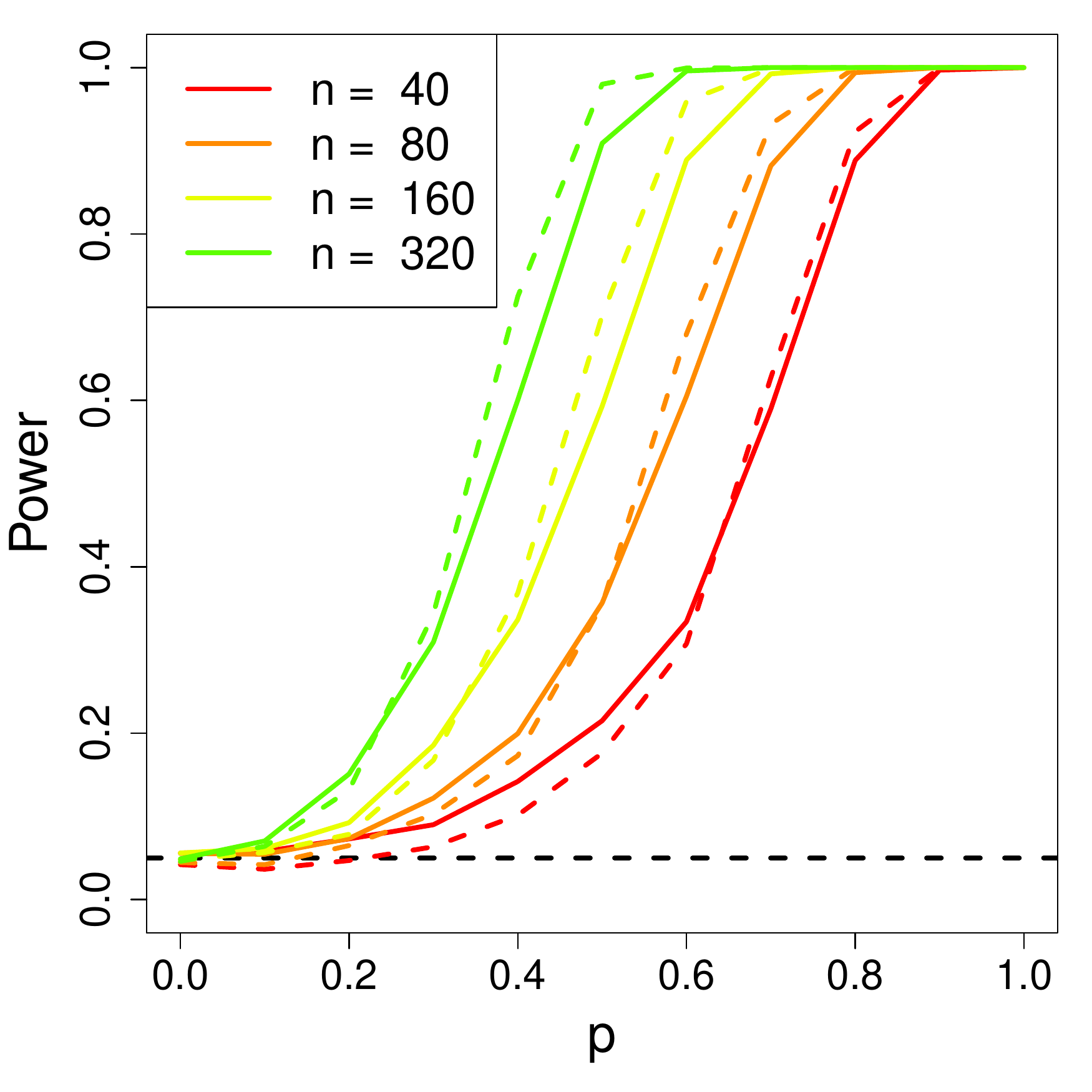}
	}
	~ \\
        	\subfloat[Case (ii)(b)\label{figure: case (ii)(b)}]{%
         		\includegraphics[width=.35\textwidth]{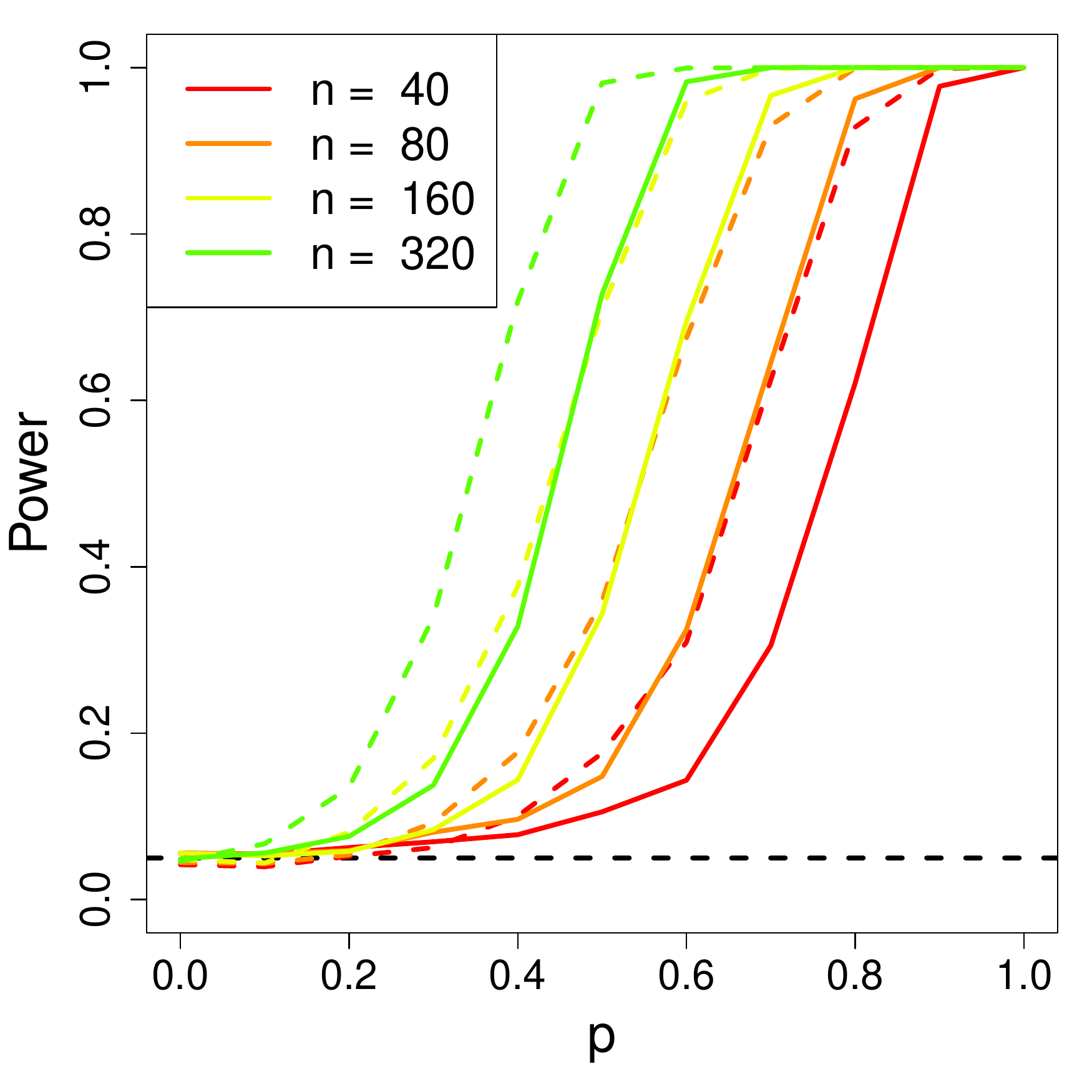}
	}
	\subfloat[Case (ii)(c)\label{figure: case (ii)(c)}]{%
         		\includegraphics[width=.35\textwidth]{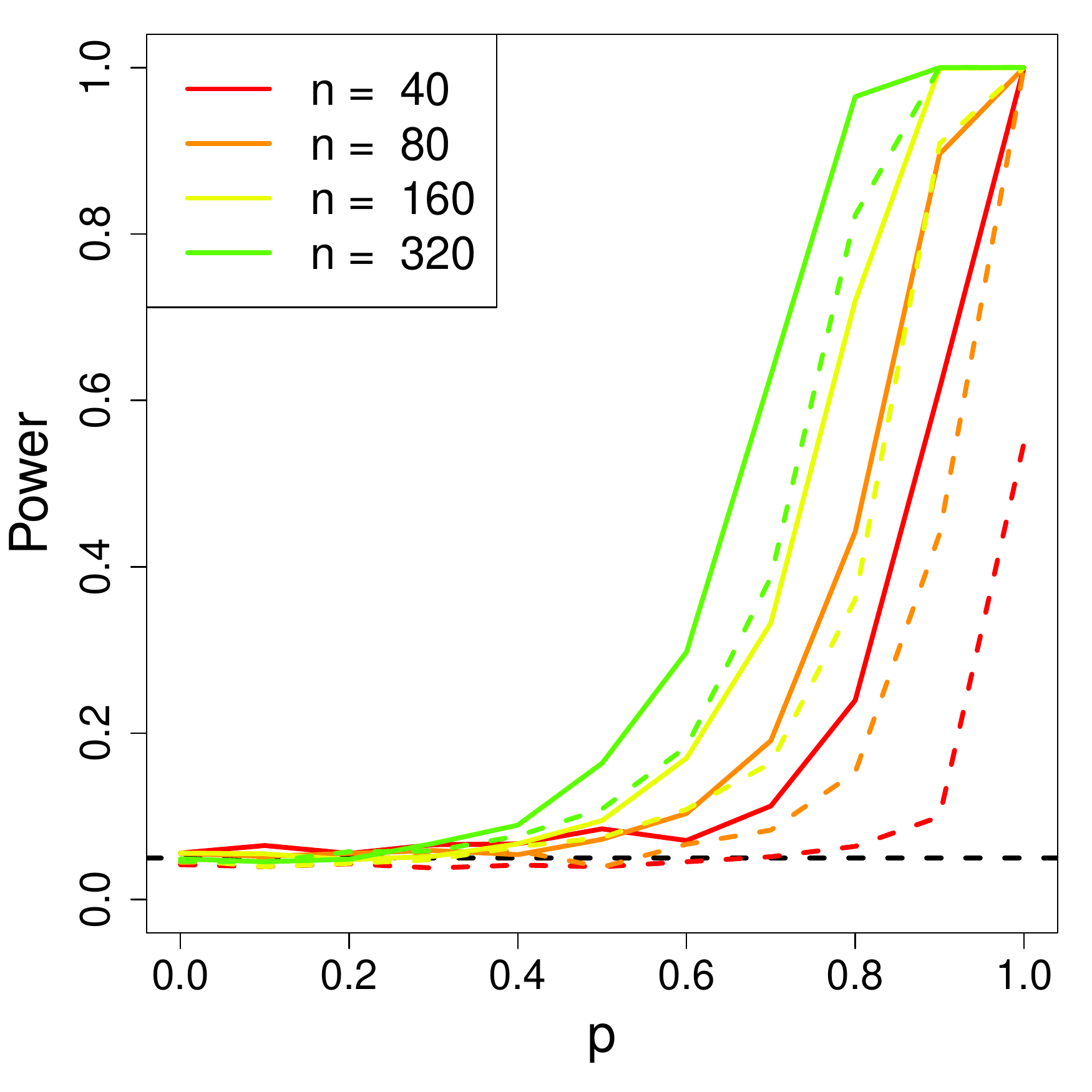}
	}
	~ \\
	\subfloat[Case (iii)(a)\label{figure: case (iii)(a)}]{%
         		\includegraphics[width=.35\textwidth]{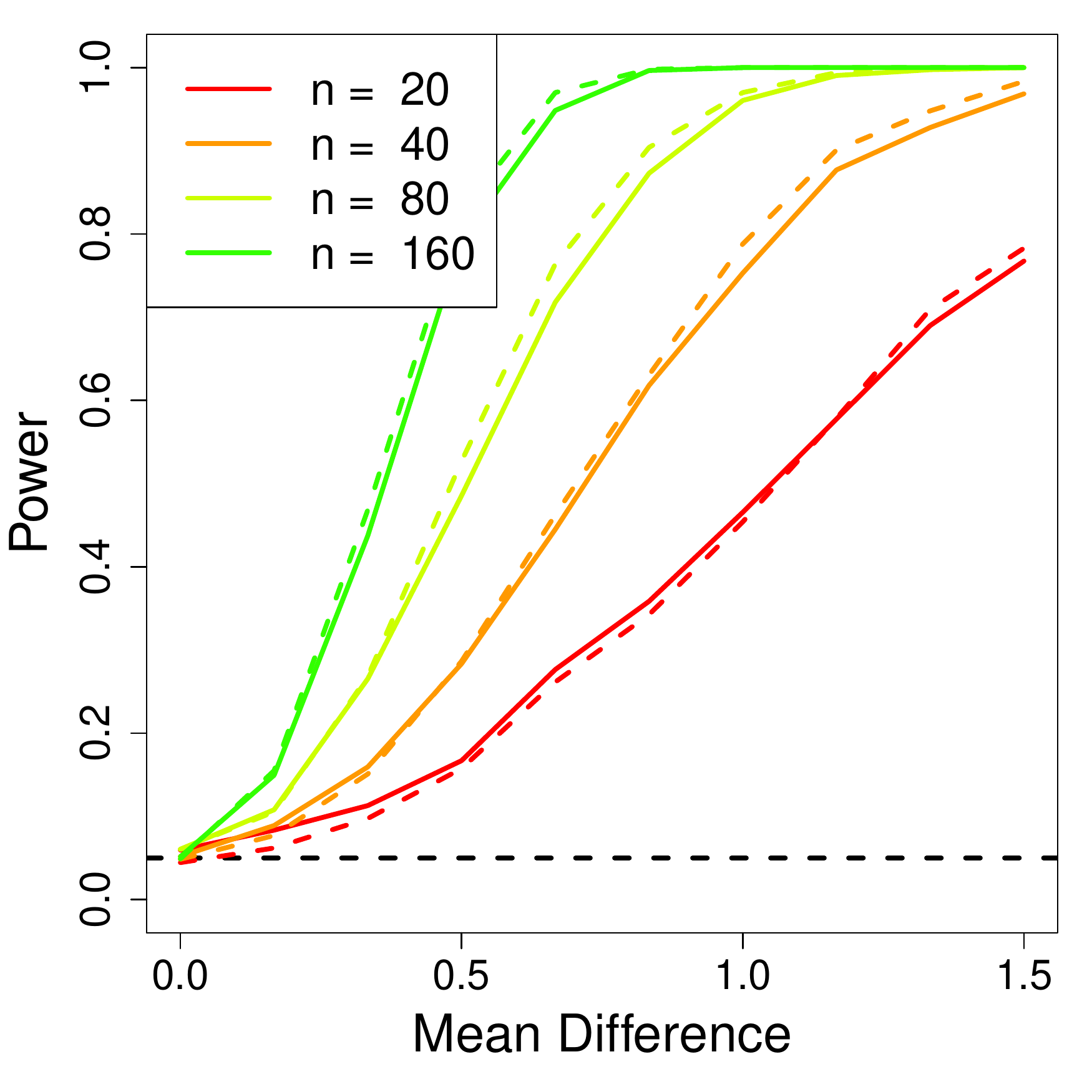}
	}
	\subfloat[Case (iii)(b)\label{figure: case (iii)(b)}]{%
         		\includegraphics[width=.35\textwidth]{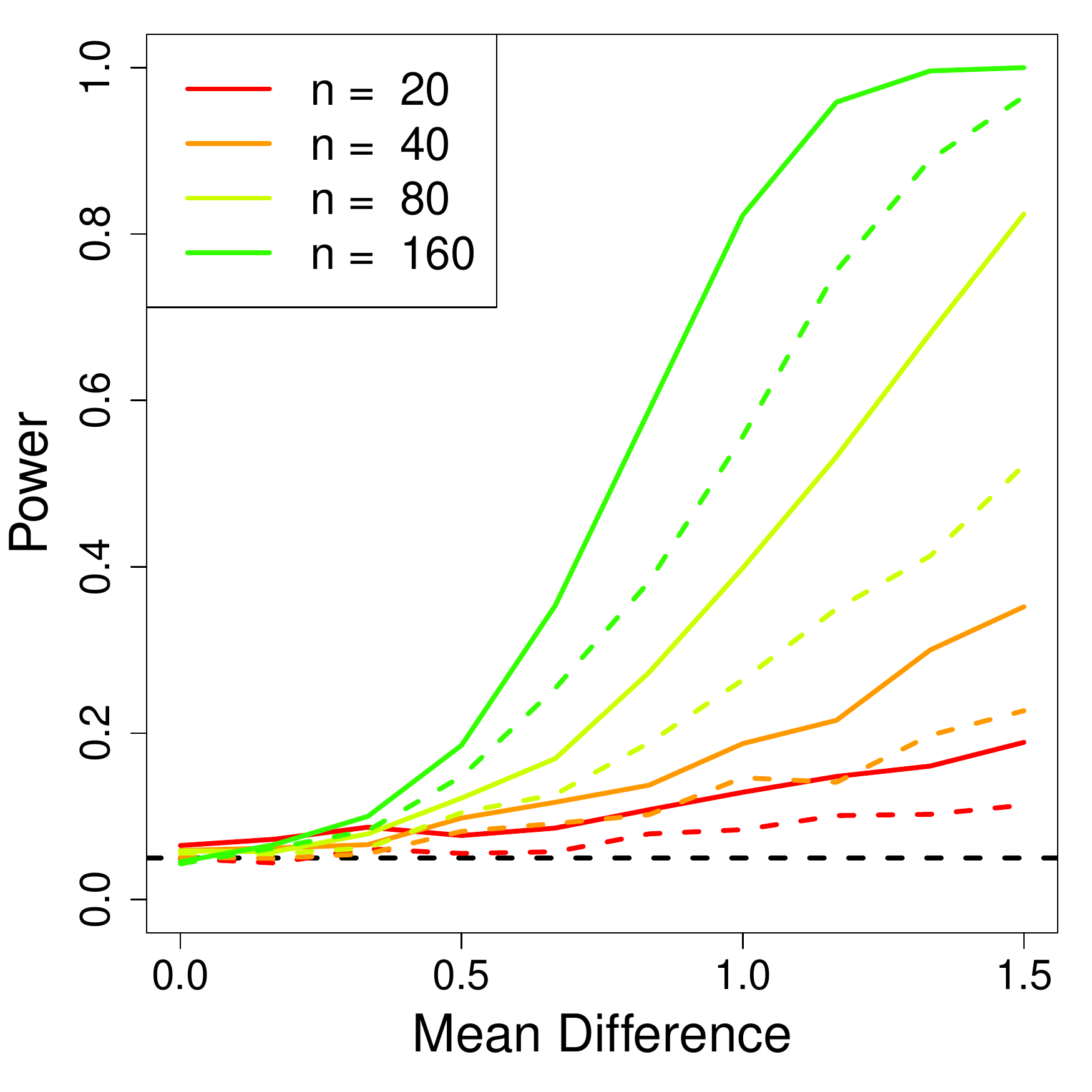}
	}
        \caption{Simulated power of the $t^*$ asymptotic test (in
          solid line) and the power of the competing test (in dashed
          line). In each case we use a level of 0.05, which is
          displayed as a horizontal dashed black line.} \label{figure: power
          simulations}
\end{figure}

\subsection{Sample size calculations} \label{subsection: computing power}

Focusing on the continuous case, consider an asymptotic level $\alpha$
test of the null hypothesis of $\tau^*=0$ (i.e., independence) that
compares the statistic $t^*$ to a critical value $c_{\alpha}$ derived
from the asymptotic distribution from Theorem \ref{theorem: continuous
  asymptotic distribution}.  Suppose we would like to determine
the minimum sample size $n_\beta$ needed for a power of at least
$\beta$ under an alternative that has the two considered variables $X$
and $Y$ dependent, so that $\tau^*>0$.  If the sample is drawn from a
joint distribution for $(X,Y)$ that satisfies the conditions of
Corollary \ref{cor: asymptotics under dependence}, and if $\sigma_1^2$
is known to be no larger than the quantity $\bar\sigma_1^2$, then
Corollary \ref{cor: asymptotics under dependence} implies that for any
$x \leq \tau^*$,
\begin{align*}
  P(t^*\leq x) &= P\left(\sqrt{n}(t^*-\tau^*) \leq \sqrt{n}(x-\tau^*)\right) \\
               &\approx P\left(N(0,16\, \sigma_1^2) \leq
                 \sqrt{n}(x-\tau^*)\right) \\ 
               &\leq P\left(N\left(\tau^*, {16\, \bar\sigma_1^2
                 }/{n}\right) \leq x\right). 
\end{align*}
This result can be used to find an asymptotically valid upper bound
$\bar n_\beta$ on $n_\beta$, by letting $\bar n_\beta$ be the smallest
positive integer such that
$c_{\alpha}/\bar n_\beta \leq \tau^*$ and
$P(N(\tau^*, {16\, \bar\sigma_1^2}/{\bar n_\beta}) \leq
c_\alpha/\bar n_\beta) \leq 1-\beta$. Finding this number
$\bar n_\beta$ can be accomplished in an iterative fashion.

The remaining difficulty in such an asymptotic sample size calculation
is finding a suitable upper bound $\bar\sigma_1^2$ for the unknown
variance $\sigma_1^2$. A crude but universally valid upper bound for
$\sigma_1^2$ can be obtained from Lemma \ref{lemma: kernel values},
which implies that $h_1(X,Y)$ takes values in the interval
$[-1/3,2/3]$ and thus $\sigma_1^2 \leq 1/4$. When $(X,Y)$ is bivariate normal,
an approximately valid upper bound of $\sigma_1^2$ is given by
$\sigma_1^2 \leq 0.14/16 = 0.00875$ (see Example \ref{example: variance
  of h1}). Figure \ref{figure: power upper bound} plots the
upper bound for the minimum sample size needed to achieve various powers when
bounding $\sigma_1^2$ by $1/4$ and $0.00875$ respectively and sampling
from a bivariate normal distribution with correlation 0.6. From the figure, we see that the $1/4$ 
bound leads to very conservative sample sizes while the $0.00875$ bound
results in values that much more closely adhere to the empirical
truth.  In general, overestimation of $\sigma_1^2$ is advisable as
small values of $\bar\sigma_1^2$ may lead to consideration of sample
sizes that are too small for asymptotic approximations to be reflective
of the actual finite-sample behavior of the test.

\begin{figure}
        \centering
        \includegraphics[width=.75\textwidth]{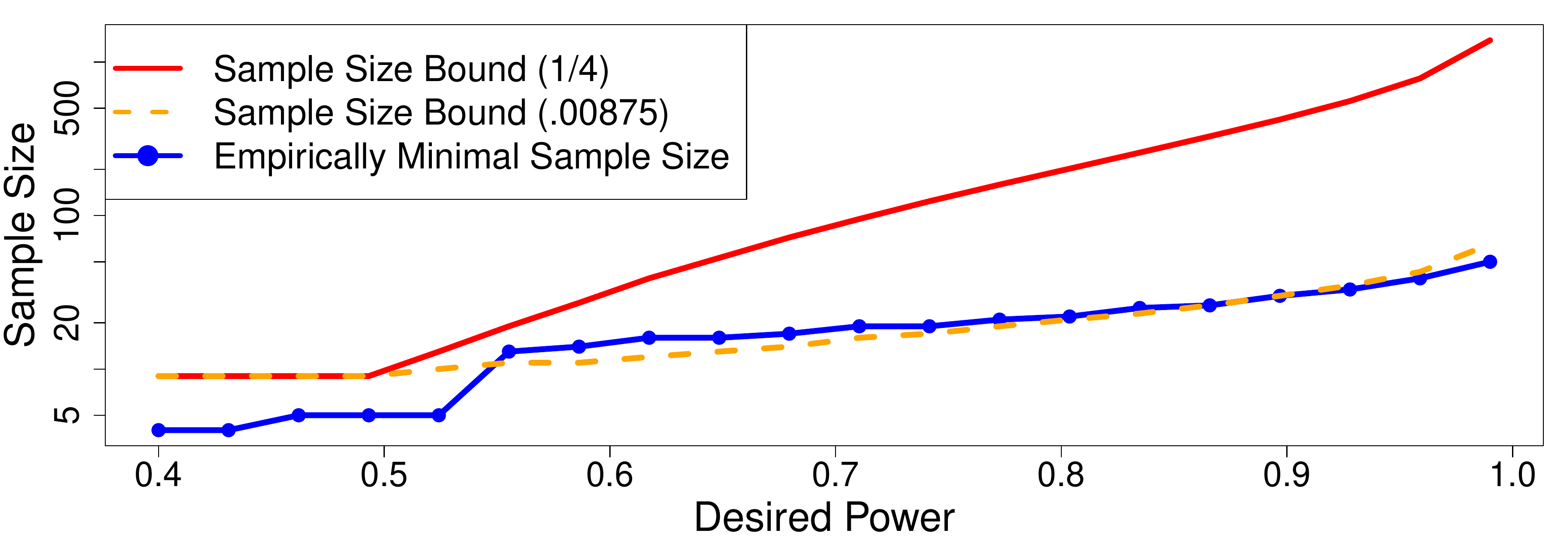}
        \caption{Minimum sample size ($n_{\beta}$) needed to a achieve a desired
          power $\beta$ at level $0.05$.  Simulations for bivariate normal
          data with correlation $0.6$ were used to compute an estimate
          of $n_{\beta}$ (blue line with dots).  These are
          compared to two asymptotic upper bounds for $n_{\beta}$
          using the bound $\sigma_1^2\leq 1/4$ (red line) and the
          bound $\sigma_1^2 \leq 0.00875$ (dashed orange line).  The
          sample size is presented with a log scaling.} \label{figure:
          power upper bound}
\end{figure}


\section{Discussion}\label{section: discussion}

The sign covariance $\tau^*$ of \cite{BergsmaDassios14} has the
intriguing property of being zero if and only if the considered pair
of random variables is independent, assuming that the random variables
follow a distribution that is continuous, discrete or a mixture of
such distributions.  Under these mild conditions, testing the
hypothesis that $\tau^*=0$ thus allows one to consistently assess
(in-)dependence.  With the aim of simplifying the implementation of
such independence tests, we have given a comprehensive study of the
asymptotic properties of $t^*$, the natural U-statistic for $\tau^*$.
The asymptotic distribution of $t^*$, especially as described in
Section \ref{subsection: continuos}, is seen to be connected in
interesting ways to the asymptotic distribution of Hoeffding's $D$,
and the Cram\'{e}r-von Mises statistic.

One limitation of our work is that we did not consider asymptotic
distributions under local alternatives to independence.  The reason is
that these would be distributions of weighted sums of non-central
chi-square random variables, which seem difficult to use in numerical
computations for assessment of power or sample size calculation.

While we have a complete understanding of the asymptotics of $t^*$
under fairly weak distributional assumptions---we covered continuous
and discrete cases, it remains to be seen if the large-sample
distribution of $t^*$ can be obtained without any such assumptions.
However, as noted above, it is also not yet known if the property that
$\tau^*=0$ only under independence holds for distributions that
are not continuous, discrete or a mixture of two such distributions.


\appendix
\section{Proofs for Section~\ref{sec:degeneracy}}\label{section: proofs-degen}

\begin{proofof}[Lemma \ref{lemma: degeneracy of h_1}]
  We show that $h_1(x_1,y_1) = 0$ for any $(x_1,y_1)$ in the support
  of $(X,Y)$. Since $X$ and $Y$ are independent and $X_1,\ldots,X_4$
  as well as $Y_1,\ldots,Y_4$ are i.i.d.\ random variables, we have
\begin{align*}
h_1((x_1,y_1)) &= \frac{1}{4!} \sum_{\pi \in S_4}  \Exp\left[  a(X_{\pi(1,2,3,4)}) \mid X_1 = x_1 \right] \Exp \left[  a(Y_{\pi(1,2,3,4)}) \mid Y_1 = y_1 \right] \\
 &= \frac{1}{4!} \sum_{\pi \in S_4} \Exp \left[  a(X_1,X_2,X_3,X_4) \mid X_{\pi(1)} = x_1\right]  \Exp \left[  a(Y_1,Y_2,Y_3,Y_4) \mid Y_{\pi(1)} = y_1 \right] .
\end{align*}
Thus it suffices to show that $g_X^{(j)}(x_1) := \Exp \left[  a(X_1,X_2,X_3,X_4) \mid X_j = x_1\right] = 0$, for $j=1,2,3,4$.
For $j=1$, we have
\begin{align*}
g_X^{(1)}(x_1) &= P(x_1, X_3 < X_2, X_4) + P(x_1, X_3 > X_2, X_4) \\
 &\quad - P(x_1, X_2 < X_3, X_4) - P(x_1, X_2 > X_3, X_4)\\
 &= 0
\end{align*}
because $X_2,X_3,X_4$ are i.i.d.\ and thus exchangeable.  Analogous arguments show that all other
$g_X^{(j)}(x_1)$ are zero. \qed
\end{proofof}
\vspace{5mm}



\begin{proofof}[Theorem \ref{theorem: nondegeneracy of h_1}]
  Let $F$ be the, by assumption, continuously differentiable joint
  distribution function of $(X,Y)$, and let $F_X$ and $F_Y$ be the two
  marginal distribution functions.  Since $h$ is invariant to
  monotonically increasing transformations of its coordinates, we may
  assume without loss of generality that we have applied $F_X$ and
  $F_Y$ to $(X,Y)$ coordinate-wise, so that $X$ and $Y$ are
  Uniform(0,1) marginally. Moreover, since we assumed that $(X,Y)$ had
  support $[a,b]\times [c,d]$ for $a<b,\ c<d$, it follows that $(X,Y)$
  has support $[0,1]^2$ after the transformation. The main idea of the
  proof is to show
\begin{align}\label{eq: nonzero partial derivative}
\left. \frac{\partial^2}{\partial y_1 \partial x_1} h_1(x_1,y_1) \right|_{(x^*,y^*)} \neq 0  ~\text{for some $(x^*,y^*) \in (0,1)^2$}.
\end{align} 
A continuity argument and \eqref{eq: nonzero partial derivative} then
imply that $h_1(x,y)$ is a non-constant function on a set of non-zero
probability and thus $h_1(X,Y)$ is non-degenerate.

Note that since $X,Y\sim\text{Uniform}(0,1)$ marginally we have that
$F_X(x) = x$ and $F_Y(y) = y$ for all $x,y\in[0,1]$.  Consequently,
the marginal densities of $X$ and $Y$,
$f_X(x) := \frac{\partial}{\partial x} F_X(x)$ and
$f_Y(y) := \frac{\partial}{\partial y} F_Y(y)$, equal 1 on $[0,1]$.
We write $f$ for the probability density function of $(X,Y)$, which is
assumed continuous, and we denote the conditional distribution
function of $X$ given $Y=y$ by $F_{X|y}(x)$ and denote the conditional
distribution function of $Y$ given $X=x$ by $F_{Y|x}(y)$.
In Lemma~\ref{lem:appendix} below, we find that
\begin{align}\label{eq: partial derivative formula}
\frac{\partial^2\, h_1(x_1,y_1)}{\partial y_1 \partial x_1} &=6G(x_1,y_1) [2f(x_1,y_1) + 1] + 6[F_{Y|x_1}(y_1) - y_1] [F_{X|y_1}(x_1) - x_1],
\end{align}
where
$G(x_1,y_1) = F(x_1,y_1) - F_X(x_1)F_Y(y_1) = F(x_1,y_1) - x_1y_1$.
We proceed to show how to derive \eqref{eq: nonzero partial
  derivative} from \eqref{eq: partial derivative
  formula}. 

Since
$\frac{\partial}{\partial x} F(x,y) = F_{Y|x}(y) f_X(x) = F_{Y\mid
  x}(y)$ and similarly
$\frac{\partial}{\partial y} F(x,y) = F_{X|y}(x)$, we have
$\frac{\partial}{\partial x_1} G(x_1,y_1) = F_{Y|x_1}(y_1) - y_1$ and
$\frac{\partial}{\partial y_1} G(x_1,y_1) = F_{X|y_1}(x_1) -
x_1$. Thus
\begin{align*}
 \frac{\partial^2\ h_1(x_1,y_1)}{\partial y_1 \partial x_1} &= 6 G(x_1,y_1) [2f(x_1,y_1) + 1] + 6 \left[ \frac{\partial}{\partial x_1} G(x_1,y_1)\right] \left[ \frac{\partial}{\partial y_1} G(x_1,y_1)\right].
\end{align*}
Now, $G$ is continuous because $F$ is, and thus the compactness of
$[0,1]^2$ yields that $G$ attains its extrema on $[0,1]^2$.  In other
words, there exist $z_m = (x_m,y_m), z_M = (x_M,y_M)\in[0,1]^2$ such
that
$G(z_m) = \inf_{(x,y)\in[0,1]^2}G(x,y),\ G(z_M) =
\sup_{(x,y)\in[0,1]^2}G(x,y)$.  Since $X$ and $Y$ are dependent we
must have that either $G(z_M) > 0$ or $G(z_m)< 0$. Without loss of
generality assume that $G(z_M) > 0$. 

The support of $(X,Y)$ being equal to $[0,1]^2$, we have that
$G(x,y) = 0$ for all $(x,y)$ on the boundary of $[0,1]^2$.  Hence,
$z_M= (x_M,y_M)$ lies in the interior of $[0,1]^2$ and as a local (global)
maximum of $G$, it satisfies
 \begin{align*}
	\left. \frac{\partial}{\partial x_1} G(x_1,y_1) \right|_{(x_M,y_M)} = \left. \frac{\partial}{\partial y_1} G(x_1,y_1) \right|_{(x_M,y_M)}  = 0.
\end{align*}
We deduce that \eqref{eq: nonzero partial
  derivative} because
\begin{align*}
\left. \frac{\partial^2}{\partial y_1 \partial x_1} h_1(x_1,y_1) \right|_{(x_M,y_M)} =  6 G(x_M,y_M) [2f(x_M,y_M) + 1] >0.
\end{align*}
(If instead we had assumed that $G(z_m) < 0$ then the same arguments
would hold and the above inequality would be $<0$ instead of $>0$.)

Finally, $\frac{\partial^2}{\partial y_1 \partial x_1} h_1(x_1,y_1)$
is easily seen to be continuous and thus
$ \frac{\partial^2}{\partial y_1 \partial x_1} h_1(x_1,y_1)>0$ in an
open neighborhood $U$ of $z_M$. Since the support of $f(x,y)$ is all
of $[0,1]^2$, that is $\ol{f^{-1}((0, \infty))} = [0,1]^2$, it follows
that $U \cap f^{-1}((0, \infty))$ is a non-empty open set and thus the
claim of the theorem follows.  \qed
\end{proofof}

\begin{lemma}
  \label{lem:appendix}
  Let $(X,Y)$ have joint density $f$ and joint distribution function
  $F$.  Let $F_X$ and $F_Y$ be the marginal distribution functions,
  and let $F_{X|y}$ and $F_{Y|x}$ be the conditional distribution
  functions of $X$ given $Y=y$ and $Y$ given $X=x$, respectively.   If
  $X,Y\sim\text{Uniform}(0,1)$ marginally,  then
  \begin{align*}
    \frac{\partial^2\, h_1(x_1,y_1)}{\partial y_1 \partial x_1} &=6G(x_1,y_1) [2f(x_1,y_1) + 1] + 6[F_{Y|x_1}(y_1) - y_1] [F_{X|y_1}(x_1) - x_1],
  \end{align*}
  where $G(x_1,y_1) = F(x_1,y_1) - F_X(x_1)F_Y(y_1) = F(x_1,y_1) - x_1y_1$.
\end{lemma}

\begin{proof}
  Let $Z_1 = (X_1,Y_1),...,Z_4 = (X_4,Y_4)$ be i.i.d.~copies of
  $(X,Y)$.  In the continuous case, $Z_1,...,Z_4$ are almost surely
  either concordant or discordant.  It follows from Lemma \ref{lemma:
    kernel values} that
  \[
    h(Z_1,Z_2,Z_3,Z_4) = I(Z_1,Z_2,Z_3,Z_4~\text{are concordant})
    - 1/3,
  \]
  where $I(\cdot)$ is the indicator function as usual.  Let
  $C(Z_1,...,Z_4)$ denote the event that $Z_1,...,Z_4$ are
  concordant.  Then
  \begin{align*}
    h_1(x_1,y_1) + 1/3 
    &= P(C(z_1,Z_2,Z_3,Z_4))  \\
    &= 3 P(x_1\leq X_2 \leq X_3, X_4 ~\text{and}~ C(z_1,Z_2,Z_3,Z_4)) \\
    &\quad+ 3 P(X_2\leq x_1 \leq X_3, X_4 ~\text{and}~ C(z_1,Z_2,Z_3,Z_4)) \\
    &\quad+ 3 P(X_3,X_4 \leq x_1 \leq X_2 ~\text{and}~ C(z_1,Z_2,Z_3,Z_4)) \\
    &\quad+ 3 P(X_3, X_4 \leq X_2 \leq x_1~\text{and}~ C(z_1,Z_2,Z_3,Z_4)).
  \end{align*}
  We make the definitions
  \begin{align*}
    P_{bl}(x,y) &:= P(X\leq x, Y \leq y), & P_{tl}(x,y)&:=P(X\leq x, Y
                                                         > y),\\
    P_{br}(x,y) &:= P(X > x, Y \leq y), & P_{tr}(x,y) &:= P(X > x, Y > y).
  \end{align*}
  As suggested by the notation, $P_{bl}(x,y)$ is the probability of
  $(X,Y)$ being in the 'bottom left' quadrant when dividing $\bR^2$ by
  the lines $\{x\}\times\bR$ and $\bR\times \{y\}$, and the notation
  for the other three probabilities is motivated similarly.  Now note
  that
\begin{align*}
	P(x_1\leq X_2 \leq &X_3, X_4 ~\text{and}~ C(z_1,Z_2,Z_3,Z_4)) \\
	& = \int_{0}^{1} \int_{x_1}^{1} [P(X_3, X_4 > x~\text{and}~ Y_3,Y_4 \leq \min(y_1,y)) \\
	&\hspace{4mm} + P(X_3, X_4 > x~\text{and}~ Y_3,Y_4 > \max(y_1,y))] f(x,y) \dx \dy \\
	& = \int_{0}^{y_1} \int_{x_1}^1 \{P_{br}^2(x,y) + P_{tr}^2(x,y_1) \} f(x,y) \dx \dy  \\
	&\hspace{4mm} +  \int_{y_1}^1 \int_{x_1}^1 \{ P_{br}^2(x,y_1) + P_{tr}^2(x,y) \} f(x,y) \dx \dy.
\end{align*}

Similarly, we have
\begin{allowdisplaybreaks}
\begin{align*}
	P(X_2&\leq x_1 \leq X_3, X_4 ~\text{and}~ C(z_1,Z_2,Z_3,Z_4)) \\
	& = \int_{0}^{y_1} \int_{0}^{x_1} \{P_{br}^2(x_1,y) + P_{tr}^2(x_1,y_1) \} f(x,y) \dx \dy  \\
	&\hspace{4mm} +  \int_{y_1}^{1} \int_{0}^{x_1} \{P_{br}^2(x_1,y_1) + P_{tr}^2(x_1,y) \} f(x,y) \dx \dy, \\
	P(X_3&,X_4 \leq x_1 \leq X_2 ~\text{and}~ C(z_1,Z_2,Z_3,Z_4)) \\
	& = \int_{0}^{y_1} \int_{x_1}^1 \{P_{bl}^2(x_1,y) + P_{tl}^2(x_1,y_1) \} f(x,y) \dx \dy  \\
	&\hspace{4mm} +  \int_{y_1}^1 \int_{x_1}^1 \{ P_{bl}^2(x_1,y_1) + P_{tl}^2(x_1,y) \} f(x,y) \dx \dy, \\
 \intertext{and}
	P(X_3&, X_4 \leq X_2 \leq x_1~\text{and}~ C(z_1,Z_2,Z_3,Z_4)) \\
	& = \int_{0}^{y_1} \int_{0}^{x_1} \{P_{bl}^2(x,y) + P_{tl}^2(x,y_1) \} f(x,y) \dx \dy  \\
	&\hspace{4mm} +  \int_{y_1}^{1} \int_{0}^{x_1} \{P_{bl}^2(x,y_1) + P_{tl}^2(x,y) \} f(x,y) \dx \dy.
\end{align*}
\end{allowdisplaybreaks}
Now, a straightforward but lengthy computation shows that
\begin{align}
	\frac{\partial^2}{\partial y_1 \partial x_1}&(\frac{1}{3}h_1(x_1,y_1) + \frac{1}{9}) \label{eq: h1 decomposition}\\
	&= \left\{\frac{\partial^2}{\partial y_1\partial x_1}P_{bl}^2(x_1,y_1)\right\} P_{tr}(x_1,y_1) + \left\{\frac{\partial^2}{\partial y_1\partial x_1}P_{tl}^2(x_1,y_1)\right\} P_{br}(x_1,y_1) \nonumber\\
	&\quad + \left\{\frac{\partial^2}{\partial y_1\partial x_1}P_{br}^2(x_1,y_1)\right\} P_{tl}(x_1,y_1) + \left\{\frac{\partial^2}{\partial y_1\partial x_1}P_{tr}^2(x_1,y_1)\right\} P_{bl}(x_1,y_1). \nonumber
\end{align}
In terms of the distribution function, the quadrant probabilities are
\begin{align*}
	P_{bl}(x_1,y_1) &= F(x_1,y_1), \\
	P_{tl}(x_1,y_1) &= F_X(x_1) - F(x_1,y_1) = x_1- F(x_1,y_1), \\
	P_{br}(x_1,y_1) &= F_Y(y_1) - F(x_1,y_1) = y_1 - F(x_1,y_1), \text{ and}\\
	P_{tr}(x_1,y_1) &= 1 - F_X(x_1) - F_Y(y_1) + F(x_1,y_1) = 1 - x_1-y_1+F(x_1,y_1).
\end{align*}
Using that $\frac{\partial}{\partial x} F(x,y) = F_{Y|x}(y)$,
$\frac{\partial}{\partial y} F(x,y) = F_{X|y}(x)$ and
$\frac{\partial^2}{\partial y \partial x} F(x,y) = f(x,y)$, we obtain that
\begin{align}\label{eq: derivative of g_{bl}}
	\bigg\{&\frac{\partial^2}{\partial y_1\partial x_1}P_{bl}^2(x_1,y_1)\bigg\} P_{tr}(x_1,y_1)\nonumber  \\
	&= 2 P_{tr}(x_1,y_1)  \left[P_{bl}(x_1,y_1) \frac{\partial^2}{\partial y_1\partial x_1}P_{bl}(x_1,y_1) +  \left\{\frac{\partial}{\partial x_1}P_{bl}(x_1,y_1)\right\}\left\{\frac{\partial}{\partial y_1}P_{bl}(x_1,y_1)\right\} \right] \nonumber \\
	&= 2 P_{tr}(x_1,y_1) \left[ P_{bl}(x_1,y_1) f(x_1,y_1) + F_{Y|x_1}(y_1)F_{X|y_1}(x_1)\right]. 
\end{align}
Similarly
\begin{allowdisplaybreaks}
  \begin{align}
    \bigg\{&\frac{\partial^2}{\partial y_1\partial x_1}P_{tl}^2(x_1,y_1)\bigg\} P_{br}(x_1,y_1)\nonumber  \\
           &= -2 P_{br}(x_1,y_1) \left[ P_{tl}(x_1,y_1) f(x_1,y_1) + (1 - F_{Y|x_1}(y_1))F_{X|y_1}(x_1) \right], \label{eq: derivative of g_{tl}} \\
               \bigg\{&\frac{\partial^2}{\partial y_1\partial x_1}P_{br}^2(x_1,y_1)\bigg\} P_{tl}(x_1,y_1) \nonumber\\ 
           &= -2 P_{tl}(x_1,y_1)[P_{br}(x_1,y_1) f(x_1,y_1)
             +F_{Y|x_1}(y_1)(1-F_{X|y_1}(x_1))], \label{eq: derivative of
             g_{br}} \\ 
    \intertext{and}
    \bigg\{&\frac{\partial^2}{\partial y_1\partial x_1}P_{tr}^2(x_1,y_1)\bigg\} P_{bl}(x_1,y_1) \nonumber \\
           &= 2 P_{bl}(x_1,y_1)[P_{tr}(x_1,y_1) f(x_1,y_1)+ (1 - F_{Y|x_1}(y_1))(1-F_{X|y_1}(x_1))] \label{eq: derivative of g_{tr}}.
  \end{align}
\end{allowdisplaybreaks}
Combining \eqref{eq: h1 decomposition}-\eqref{eq: derivative of g_{tr}},
we find that
\begin{align*}
	&\frac{1}{3} \frac{\partial^2}{\partial y_1 \partial x_1} h(x_1,y_1) \\
	&= 4 [P_{bl}(x_1,y_1)P_{tr}(x_1,y_1) - P_{tl}(x_1,y_1)P_{br}(x_1,y_1)] f(x_1,y_1) \\
	&\quad +  2 [F(x_1,y_1) + F_{Y|x_1}(y_1)F_{X|y_1}(x_1) -x_1F_{Y|x_1}(y_1) - y_1F_{X|y_1}(x_1) ] \\
	&= 2[F(x_1,y_1) - x_1y_1][2f(x_1,y_1) + 1] + 2 [F_{Y|x_1}(y_1) - y_1] [F_{X|y_1}(x_1) - x_1],
\end{align*}
which gives the claimed formula. 
\end{proof}

\section{Proofs for Section~\ref{section: asymptotics under the null}}\label{section: proofs}

\begin{proofof}[Lemma \ref{lemma: product formula for h_2}]
First note that
\begin{align}
  \nonumber
  h_2&((x_1,y_1),(x_2,y_2)) \\
  \nonumber
  &= \frac{1}{4!} \sum_{\pi \in S_4}  \Exp\left[
    a(X_{\pi(1,2,3,4)}) \mid X_1 = x_1, X_2 = x_2 \right]
  \Exp\left[  a(Y_{\pi(1,2,3,4)}) \mid Y_1 = y_1, Y_2 = y_2
  \right] \\
  \nonumber
  &= \frac{1}{4!} \sum_{\pi \in S_4} \bigg( \Exp\left[
    a(X_1,X_2,X_3,X_4) \mid X_{\pi(1)} = x_1, X_{\pi(2)} = x_2 \right]
  \\
  \nonumber
  &\hspace{18mm} \cdot \Exp\left[  a(Y_1,Y_2,Y_3,Y_4) \mid Y_{\pi(1)}
    = y_1, Y_{\pi(2)} = y_2 \right] \bigg) \\
  \label{eq:compute-h2}
  &=:  \frac{1}{4!} \sum_{\pi \in S_4} g_X^{\pi}(x_1,x_2) g_Y^{\pi}(y_1,y_2).
\end{align}
The first equality follows from the independence of $X$ and $Y$ and
the second equality follows from the fact that $X_1,\ldots,X_4$ (and $Y_1,\ldots,Y_4$) are i.i.d.\ random variables. 

Next, recall from~(\ref{eq:prob-form-of-g}) that
\begin{align*}
	 g_{X}(x_1,x_2) 
 	&\ = P(x_1, X_3 < x_2, X_4) + P(x_1, X_3 > x_2, X_4) \\
	&\hspace{6mm} - P( x_1, x_2 < X_3, X_4) - P(x_1, x_2 > X_3, X_4).
\end{align*}
We claim that
\begin{align}\label{eq: g_X}
g_X^{\pi}(x_1,x_2) = \left\{ \begin{array}{ll} g_X(x_1,x_2) & \text{if $\pi(1),\pi(2) \in \{1,2\}$ or $\pi(1),\pi(2) \in \{3,4\}$}, \\ - g_X(x_1,x_2) &  \text{if $\pi(1),\pi(2) \in \{1,3\}$ or $\pi(1),\pi(2) \in \{2,4\}$}, \\ 0 &  \text{otherwise}. \end{array} \right.
\end{align}
Note that \eqref{eq: g_X} implies that $g_X^{\pi}(x_1,x_2)$ is nonzero
for 16 of the 24 permutations $\pi\in S_4$.  For a set of 8 of these
permutations, $g_X^{\pi}(x_1,x_2)=g_X(x_1,x_2)$, and for the other 8,
$g_X^{\pi}(x_1,x_2)=-g_X(x_1,x_2)$.  The analogue is true for
$g_Y^{\pi}(y_1,y_2)$.  Taking products and summing over the
permutations $\pi$ as in~(\ref{eq:compute-h2}) completes the proof of
the formula for $h_2((x_1,y_1),(x_2,y_2))$.

It remains to show the claim in \eqref{eq: g_X}. Since $X_3$ and $X_4$
are i.i.d.\ random variables, $\pi(1),\pi(2) \in \{1,2\}$ implies
$g_X^{\pi}(x_1,x_2) = g_X(x_1,x_2)~\text{or}~g_X(x_2,x_1)$. But $g_X$
is symmetric and thus $\pi(1),\pi(2) \in \{1,2\}$ implies
$g_X^{\pi}(x_1,x_2) = g_X(x_1,x_2)$.  Analogously, it
follows that $g_X^{\pi}(x_1,x_2)
= g_X(x_1,x_2)$ if $\pi(1),\pi(2) \in \{3,4\}$ because
\begin{align*}
	\Exp[a(X_1,X_2,x_1,x_2)] &= P(X_1, x_1 < X_2, x_2) + P(X_1, x_1 > X_2, x_2) \\
	& \qquad- P(X_1, X_2 < x_1, x_2) - P(X_1, X_2 > x_1, x_2) \\
	&= g_X(x_1,x_2).
\end{align*}

Now if $\pi(1) = 1$ and $\pi(2) = 3$, then 
\begin{align*}
	g_X^{\pi}(x_1,x_2) = \Exp[a(x_1,X_2,x_2,X_4)] &= P(x_1, x_2 < X_2, X_4) + P(x_1, x_2 > X_2, X_4) \\
	& \qquad- P(x_1, X_2 < x_2, X_4) - P(x_1, X_2 > x_2, X_4) \\
	&= -g_X(x_1,x_2).
\end{align*}
Similar symmetry arguments thus yield that $g_X^{\pi}(x_1,x_2) =
-g_X(x_1,x_2)$ if $\pi(1),\pi(2) \in \{1,3\}$ or if $\pi(1),\pi(2) \in
\{2,4\}$.

In the remaining cases, we have $\pi(1),\pi(2)\in\{1,4\}$ or
$\pi(1),\pi(2)\in\{2,3\}$.  If $\pi(1)=1$ and $\pi(2)=4$,
\begin{align*}
	g_X^{\pi}(x_1,x_2) = \Exp[a(x_1,X_2,X_3,x_2)] &= P(x_1, X_3 < X_2, x_2) + P(x_1, X_3 > X_2, x_2) \\
	& \qquad- P(x_1, X_2 < X_3, x_2) - P(x_1, X_2 > X_3, x_2) \\
	&= P(x_1, X_3 < X_2, x_2) + P(x_1, X_3 > X_2, x_2) \\
	& \qquad- P(x_1, X_3 < X_2, x_2) - P(x_1, X_3 > X_2, x_2) \\
	& = 0.
\end{align*}
Similarly, $g_X^{\pi}(x_1,x_2)=0$ if $\pi(1)=4$ and $\pi(2)=1$, or if
$\pi(1),\pi(2)\in\{2,3\}$. \qed
\end{proofof}
\vspace{5mm}


\begin{proofof}[Lemma \ref{lemma: h_values}]
  Let $\phi_{i,1},\phi_{i,2},\ldots$ be the sequence of orthonormal
  eigenfunctions associated to the nonzero eigenvalues $\lambda_{i,1},
  \lambda_{i,2},\ldots$ of $A_{g_i}$, for $i=1,2$.
Since $X\indep Y$, for each $(j_1,j_2)\in\bN_+^2$,
\begin{align*}
&\Exp [k((x_1,y_1),(X_2,Y_2)) \phi_{1,j_1}(X_2)  \phi_{2,j_2}(Y_2)] \\ 
&=  \Exp [g_1(x_1,X_2)g_2(y_1,Y_2) \phi_{1,j_1}(X_2)  \phi_{2,j_2}(Y_2)] \\
&= \Exp [g_1(x_1,X_2) \phi_{1,j_1}(X_2) ] \,\Exp [g_2(y_1,Y_2) \phi_{2,j_2}(Y_2)] \\
 &= \lambda_{1,j_1} \phi_{1,j_1}(x_1) \lambda_{2,j_2} \phi_{2,j_2}(y_1).
\end{align*}
Therefore, for each $(j_1,j_2)\in\bN_+^2$,
$\lambda_{1,j_1}\lambda_{2,j_2}$ is an eigenvalue of $A_k$ with the
associated eigenfunction $\phi_{1,j_1}\phi_{2,j_2}$. Further,
$\{\phi_{1,j_1}\phi_{2,j_2}:(j_1,j_2)\in\bN_+^2\}$ is an orthonormal
system, since both $\{\phi_{1,1},\phi_{1,2},\ldots\}$ and
$\{\phi_{2,1},\phi_{2,2},\ldots\}$ are orthonormal systems, and
$X\indep Y$.

Now suppose $\{\gamma_1,\gamma_2,\ldots\}$ is a sequence of all
nonzero eigenvalues of $A_k$ with the associated orthonormal
sequence of eigenfunctions $\{\psi_1,\psi_2,\ldots\}$. Then
\[
\sum_{j=1}^n \gamma_j \psi_j((X_1,Y_1),(X_2,Y_2))
\overset{L^2}{\longrightarrow} k((X_1,Y_1),(X_2,Y_2)).
\]
By independence, 
\begin{align*}
	\sum_{j_1=1}^{\infty}\sum_{j_2=1}^{\infty} \lambda_{1,j_1}^2\lambda_{2,j_2}^2 &= \Exp[g_1(X_1,X_2)^2]~ \Exp[g_2(Y_1,Y_2)^2] \\
	&=\Exp[(g_1(X_1,X_2)g_2(Y_1,Y_2))^2] \\
	&= \Exp[k((X_1,Y_1),(X_2,Y_2))^2] \\
	&= \sum_{j=1}^{\infty} \gamma_j^2 
\end{align*}
Therefore,
we conclude that, as a multi-set,
$\{\lambda_{1,j_1}\lambda_{2,j_2}:(j_1,j_2)\in\bN_+^2\}$ contains all
nonzero eigenvalues of $A_k$ with the correct multiplicity. \qed
\end{proofof}
\vspace{5mm}


\begin{proofof}[Lemma \ref{lemma: continuous factorization}]
  Any collection of i.i.d.~continuous random variables has their rank
  vector following a uniform distribution.  Since the function $a$
  from~(\ref{eq:def-a}) depends on its arguments only through their ranks,
  we have that 
  \[
    g_{X}(x_1,x_2) = \Exp[a(x_1,x_2,X_3,X_4)]=
    \Exp[a(x_1,x_2,Y_3,Y_4)] = g_Y(x_1,x_2),  \quad x_1,x_2\in(0,1).
  \]
  Applying Lemma \ref{lemma: product formula for h_2}, we have that 
  \[
    h_2((x_1,x_2),(y_1,y_2) )= \frac{2}{3}~g_{X}(x_1,x_2) ~g_{X}(y_1,y_2)
  \]
  and the proof is complete once the following claim is established:
  \begin{equation}
    \label{eq:gX=3c}
    g_X(x_1,x_2) = -3c(x_1,x_2),   \quad x_1,x_2\in(0,1).
  \end{equation}

  Letting $x_{(1)} = x_1\wedge x_2=\min\{x_1,x_2\}$ and
  $x_{(2)} = x_1\vee x_2=\max\{x_1,x_2\}$, we have
\begin{align*}
	g_X(x_1,x_2) 
	&= P(x_1, X_3 < x_2, X_4) + P(x_1, X_3 > x_2, X_4) \\ & \qquad - P( x_1, x_2 < X_3, X_4) - P(x_1, x_2 > X_3, X_4) \\
	&= P(x_{(1)}, X_3 < x_{(2)}, X_4) - P(x_{(2)} < X_3, X_4) - P(x_{(1)} >  X_3, X_4) \\
	&= P(x_{(1)}, X_3 < x_{(2)}, X_4) - (1-x_{(2)})^2 - x_{(1)}^2.
\end{align*}
Moreover,
\begin{allowdisplaybreaks}
\begin{align*}
	P(x_{(1)}, X_3 < x_{(2)}, X_4) &= P(x_{(1)}< X_4\text{ and } X_3<x_{(1)}) \\
	&\hspace{5mm} + P(X_3<X_4\text{ and } x_{(1)}<X_3<x_{(2)}) \\
	&= x_{(1)}(1-x_{(1)}) + \int_{x_{(1)}}^{x_{(2)}}P(x<X_4 \mid X_3=x)\dx \\
	&= x_{(1)}(1-x_{(1)}) + \int_{x_{(1)}}^{x_{(2)}} (1-x)\dx \\
	&= x_{(1)}(1-x_{(1)}) + x_{(2)}\left(1-\frac{1}{2}x_{(2)}\right) - x_{(1)}\left(1-\frac{1}{2}x_{(1)}\right).
\end{align*}
\end{allowdisplaybreaks}
We obtain that
\begin{align*}
	g_X(x_1,x_2) &= -1-\frac{3}{2}x_{(1)}^2-\frac{3}{2} x_{(2)}^2 + 3x_{(2)} = -1-\frac{3}{2}x_{1}^2-\frac{3}{2} x_{2}^2 + 3x_{(2)},
\end{align*}
which is the claim from~(\ref{eq:gX=3c}).
\qed
\end{proofof}
\vspace{5mm}


\begin{proofof}[Theorem \ref{theorem: continuous asymptotic
    distribution}]
  Let $c$ be the kernel function for the Cram\'{e}r-von Mises
  statistic, as defined in Lemma \ref{lemma: continuous
    factorization}.  The operator $A_c$ is known to have eigenvalues
  $\frac{1}{j^2\pi^2}$ with corresponding eigenfunctions
  $\sqrt{2}\cos(\pi j x)$ for $j=1,2,\dots$ \cite[Example
  12.13]{VanderVaart98}.  Since
  $h_2((x_1,y_1), (x_2,y_2)) = 6 c(x_1,x_2) c(y_1,y_2)$ by Lemma
  \ref{lemma: continuous factorization}, it follows from Lemma
  \ref{lemma: h_values} that $\frac{1}{6}\, h_2$ has eigenvalues
  $\{\frac{1}{\pi^{4}}\frac{1}{j^2i^2}: (i,j)\in \bN_+^2\}$
  corresponding to orthonormal eigenfunctions
  $\{2\cos(\pi j x)\cos(\pi j y): (i,j)\in \bN_+^2\}$.  The eigenvalues
  of $h_2$ are a multiple of $6$ larger, with the same orthonormal
  eigenfunctions.  We obtain from Theorem \ref{theorem: degenerate
    asymptotics} that
\begin{align*}
	nt^* \tod {4\choose 2} \sum_{i=1}^\infty \sum_{j=1}^\infty \frac{1}{\pi^4}\frac{6}{j^2i^2}(\chi_{1,ij}^2 -1) = \frac{36}{\pi^4}\sum_{i=1}^\infty\sum_{j=1}^\infty \frac{1}{j^2i^2}(\chi_{1,ij}^2-1)
\end{align*}
where $\{\chi^2_{1,ij} : i,j \in \bN_+ \}$ is a collection of i.i.d.\ $\chi_1^2$ random variables. \qed
\end{proofof}
\vspace{5mm}


\begin{proofof}[Theorem \ref{theorem: discrete asymptotic distribution}]
Recall from Lemma \ref{lemma: product formula for h_2} that we have the factorization
\begin{align*}
	h_2((x_1,y_1),(x_2,y_2)) = \frac{2}{3}~g_{X}(x_1,x_2) ~g_{Y}(y_1,y_2).
\end{align*}
As in the proof of Theorem \ref{theorem: continuous asymptotic
  distribution}, finding the eigenvalues of $h_2$ requires only
finding the eigenvalues of the operators $A_{g_{X}}$ and $A_{g_{Y}}$. To obtain
positive eigenvalues it will be useful to instead find the eigenvalues
of $A_{-g_X}$ and $A_{-g_Y}$ which are simply the negation of the eigenvalues of
 $A_{g_{X}}$ and $A_{g_{Y}}$. For notational simplicity let $k_X = -g_X$ and $k_Y=-g_Y$.
Obtaining the eigenvalues of $A_{k_X}$ and $A_{k_Y}$ are two analogous problems
and we thus discuss only $A_{k_X}$. We denote the support of $X$ by $\{u_1,\dots,u_r\}$.

Combining the first two probabilities in~(\ref{eq:prob-form-of-g}) and
using that $X_3$ and $X_4$ are i.i.d.\ copies of $X$, the function
$k_{X}(x_1,x_2) = -\Exp[a(x_1,x_2,X_3,X_4)]$ can be written as
\begin{align*}
  k_{X}&(x_1,x_2) = P(x_1\wedge x_2,X_3 < x_1\vee x_2, X_4) - P(x_1\vee x_2 < X)^2 - P(x_1\wedge x_2 > X)^2\\
	&= \bigg[(F_X(x_1\wedge x_2) - p_X(x_1\wedge x_2))^2 + (1-F_X(x_1\vee x_2))^2\bigg] \\
	&\ \ \ -I(x_1\not=x_2)\bigg[F_X(x_1\wedge x_2)(1-F_X(x_1\wedge x_2)) +\!\! \sum_{x_1\wedge x_2 < u_\ell  < x_1\vee x_2}\!\! p_X(u_\ell)(1-F_X(u_\ell))\bigg] .
\end{align*}
Finding the eigenvalues of $A_{k_X}$ requires finding
$\lambda\in\mathbb{R}$ and a function $\phi$ such that
\begin{align}
  \label{eq:evalues-finite}
	\lambda &\phi(x) = \Exp[k_X(x, X_2)\phi(X_2)] =  \sum_{j=1}^r  p_X(u_j) \phi(u_j) k_X(x,u_j),
\end{align}
for $x$ in the support of $X$.  Since the support is finite,
(\ref{eq:evalues-finite}) is a system of $r$ linear equations in $r$
unknowns $\phi(u_1),\dots,\phi(u_r)$.  We recognize that the
eigenvalues of $A_{k_X}$ are the eigenvalues of the $r\times r$ matrix
$\tilde R^X$ whose $(i,j)$-th entry is $k_X(u_i,u_j)p_X(u_j)$.

Let $K^X$ be the symmetric $r\times r$ matrix with $(i,j)$-th entry
$k_X(u_i,u_j)$, and let $\text{diag}(p_X)$ be the diagonal $r\times r$
matrix whose diagonal entries are $p_X(u_1),...,p_X(u_r)$.  Then
$\tilde R^X=K^X\ \text{diag}(p_X)$. Noting that $\tilde R^X$ has same eigenvalues as the symmetric
matrix $R^X=$ $\text{diag}(p_X)^{1/2} K^X\ \text{diag}(p_X)^{1/2}$, we obtain that the eigenvalues of 
$A_{k_X}$ are the eigenvalues of $R^X$,
which is the matrix given by~\eqref{equation: R matrix form}.  
Since the analogous fact holds for $k_Y$, an application of
Lemma~\ref{lemma: h_values} and Theorem~\ref{theorem: degenerate
  asymptotics} completes the proof.
\qed
\end{proofof}
\vspace{5mm}


\bibliographystyle{abbrvnat}
\bibliography{bersma_bib}

\begin{thebibliography}{12}
\providecommand{\natexlab}[1]{#1}
\providecommand{\url}[1]{\texttt{#1}}
\expandafter\ifx\csname urlstyle\endcsname\relax
  \providecommand{\doi}[1]{doi: #1}\else
  \providecommand{\doi}{doi: \begingroup \urlstyle{rm}\Url}\fi

\bibitem[Bergsma and Dassios(2014)]{BergsmaDassios14}
W.~Bergsma and A.~Dassios.
\newblock A consistent test of independence based on a sign covariance related
  to {K}endall's tau.
\newblock \emph{Bernoulli}, 20\penalty0 (2):\penalty0 1006--1028, 2014.

\bibitem[Blum et~al.(1961)Blum, Kiefer, and Rosenblatt]{BlumEtAl61}
J.~R. Blum, J.~Kiefer, and M.~Rosenblatt.
\newblock Distribution free tests of independence based on the sample
  distribution function.
\newblock \emph{Ann. Math. Statist.}, 32:\penalty0 485--498, 1961.

\bibitem[Hoeffding(1948)]{Hoeffding48}
W.~Hoeffding.
\newblock A non-parametric test of independence.
\newblock \emph{Ann. Math. Statistics}, 19:\penalty0 546--557, 1948.

\bibitem[Kendall(1938)]{Kendall38}
M.~G. Kendall.
\newblock A new measure of rank correlation.
\newblock \emph{Biometrika}, 30\penalty0 (1/2):\penalty0 pp. 81--93, 1938.

\bibitem[{R Core Team}(2015)]{R}
{R Core Team}.
\newblock \emph{R: A language and environment for statistical computing}.
\newblock R Foundation for Statistical Computing, Vienna, Austria, 2015.
\newblock URL \url{https://www.R-project.org/}.

\bibitem[Rizzo and Szekely(2014)]{RizzoSzekely14}
M.~L. Rizzo and G.~J. Szekely.
\newblock \emph{energy: E-statistics (energy statistics)}, 2014.
\newblock URL \url{http://CRAN.R-project.org/package=energy}.
\newblock R package version 1.6.2.

\bibitem[Serfling(1980)]{Serfling:1980}
R.~J. Serfling.
\newblock \emph{Approximation theorems of mathematical statistics}.
\newblock John Wiley \& Sons, Inc., New York, 1980.
\newblock Wiley Series in Probability and Mathematical Statistics.

\bibitem[Spearman(1904)]{Spearman04}
C.~Spearman.
\newblock The proof and measurement of association between two things.
\newblock \emph{The American Journal of Psychology}, 15:\penalty0 72--101,
  1904.

\bibitem[Sz{\'e}kely et~al.(2007)Sz{\'e}kely, Rizzo, and
  Bakirov]{SzekelyEtAl07}
G.~J. Sz{\'e}kely, M.~L. Rizzo, and N.~K. Bakirov.
\newblock Measuring and testing dependence by correlation of distances.
\newblock \emph{Ann. Statist.}, 35\penalty0 (6):\penalty0 2769--2794, 2007.

\bibitem[van~der Vaart(1998)]{VanderVaart98}
A.~W. van~der Vaart.
\newblock \emph{Asymptotic statistics}.
\newblock Cambridge Series in Statistical and Probabilistic Mathematics.
  Cambridge University Press, Cambridge, 1998.

\bibitem[Weihs(2015)]{taustar-package}
L.~Weihs.
\newblock \emph{TauStar: Efficient computation of the t* statistic of Bergsma
  and Dassios (2014)}, 2015.
\newblock URL \url{http://CRAN.R-project.org/package=TauStar}.
\newblock {R} package version 1.0.0.

\bibitem[Weihs et~al.(2016)Weihs, Drton, and Leung]{WeihsEtAl15}
L.~Weihs, M.~Drton, and D.~Leung.
\newblock Efficient computation of the {B}ergsma-{D}assios sign covariance.
\newblock \emph{Computational Statistics}, x:\penalty0 x--x, 2016.
\newblock to appear.

\end{thebibliography}

\end{document}